\documentclass{article}
\usepackage[margin=1in,letterpaper]{geometry}

\usepackage[numbers,sort]{natbib}

\usepackage[utf8]{inputenc} \usepackage[T1]{fontenc}
\usepackage{url}
\usepackage{amsfonts}
\usepackage{nicefrac} 
\usepackage{microtype}

\usepackage{color}
\usepackage{amsmath,amsthm}
\usepackage{mathrsfs}
\usepackage{extarrows}
\usepackage{amssymb}
\usepackage{graphicx}

\usepackage[affil-sl]{authblk}
\usepackage{hyperref}
\hypersetup{colorlinks,allcolors=[rgb]{0,0.13672,0.95703}}

\title{On the number of variables to use in principal component regression}

\author{Ji Xu}
\author{Daniel Hsu}
\affil{Columbia University}
\newtheorem{theorem}{Theorem}
\newtheorem{lemma}{Lemma}
\newtheorem{proposition}{Proposition}

\theoremstyle{definition}

\theoremstyle{remark}
\newtheorem{remark}{Remark}

\newcommand\A[1]{\textbf{A.#1}}
\newcommand\B[1]{\textbf{B.#1}}

\def\vbX{\bar{\vX}}
\def\vtX{\tilde{\vX}}

\def\vtx{\tilde{\vx}}

\def\vbx{\bar{\vx}}

\def\vtS{\tilde{\vS}}
\def\vbS{\bar{\vS}}

\def\vtSigma{\tilde{\vSigma}}
\def\tsigma{\tilde{\lambda}}

\def\vXP{\vX_P}
\def\vXpC{\vX_{P^c}}
\def\vbXP{\vbX_P}
\def\vtXP{\vtX_P}

\def\vSigmaP{\vSigma_P}
\def\vSigmapC{\vSigma_{P^c}}
\def\vtSigmaP{\vtSigma_P}

\def\={\ = \ }

\def\qand{\quad \text{and} \quad }

\def\ast{\stackrel{\mathrm{a.s.}}{\rightarrow}}

\def\ipt{\stackrel{\mathrm{p}}{\rightarrow}}
\def\Ipt{\stackrel{\mathrm{p}}{\rightarrow}}

\def\BE{\begin{eqnarray}}
\def\EE{\end{eqnarray}}

\def\error{\operatorname{Error}}

\def\n{\nonumber}

\def\ratiop{\alpha}
\def\ration{\beta}

\def\op{o_{\mathrm{p}}}
\def\Op{O_{\mathrm{p}}}
\def\Omp{\Omega_{\mathrm{p}}}
\def\Thp{\Theta_{\mathrm{p}}}

\newcommand \trace[1]{\operatorname{tr}\left(#1\right)} 
\newcommand \tr[1]{\operatorname{tr}(#1)}

\newcommand\Normal{\mathcal{N}}

\usepackage{amsmath,amsbsy,amsfonts,amssymb,amsthm,dsfont,mleftright,commath}

\def\ddefloop#1{\ifx\ddefloop#1\else\ddef{#1}\expandafter\ddefloop\fi}

\def\ddef#1{\expandafter\def\csname bf#1\endcsname{\ensuremath{\mathbf{#1}}}}
\ddefloop ABCDEFGHIJKLMNOPQRSTUVWXYZabcdefghijklmnopqrstuvwxyz\ddefloop

\def\ddef#1{\expandafter\def\csname bf#1\endcsname{\ensuremath{\pmb{\csname #1\endcsname}}}}
\ddefloop {alpha}{beta}{gamma}{delta}{epsilon}{varepsilon}{zeta}{eta}{theta}{vartheta}{iota}{kappa}{lambda}{mu}{nu}{xi}{pi}{varpi}{rho}{varrho}{sigma}{varsigma}{tau}{upsilon}{phi}{varphi}{chi}{psi}{omega}{Gamma}{Delta}{Theta}{Lambda}{Xi}{Pi}{Sigma}{varSigma}{Upsilon}{Phi}{Psi}{Omega}{ell}\ddefloop

\def\ddef#1{\expandafter\def\csname bb#1\endcsname{\ensuremath{\mathbb{#1}}}}
\ddefloop ABCDEFGHIJKLMNOPQRSTUVWXYZ\ddefloop

\def\ddef#1{\expandafter\def\csname c#1\endcsname{\ensuremath{\mathcal{#1}}}}
\ddefloop ABCDEFGHIJKLMNOPQRSTUVWXYZ\ddefloop

\def\ddef#1{\expandafter\def\csname v#1\endcsname{\ensuremath{\boldsymbol{#1}}}}
\ddefloop ABCDEFGHIJKLMNOPQRSTUVWXYZabcdefghijklmnopqrstuvwxyz\ddefloop

\def\ddef#1{\expandafter\def\csname v#1\endcsname{\ensuremath{\boldsymbol{\csname #1\endcsname}}}}
\ddefloop {alpha}{beta}{gamma}{delta}{epsilon}{varepsilon}{zeta}{eta}{theta}{vartheta}{iota}{kappa}{lambda}{mu}{nu}{xi}{pi}{varpi}{rho}{varrho}{sigma}{varsigma}{tau}{upsilon}{phi}{varphi}{chi}{psi}{omega}{Gamma}{Delta}{Theta}{Lambda}{Xi}{Pi}{Sigma}{varSigma}{Upsilon}{Phi}{Psi}{Omega}{ell}\ddefloop

\renewcommand\t{{\ensuremath{\scriptscriptstyle{\top}}}}

\renewcommand{\P}{\ensuremath{\mathbb{P}}}
\DeclareMathOperator{\var}{var}

\renewcommand\v{\ensuremath{\boldsymbol}}

 \newcommand\ind[1]{\ensuremath{\mathds{1}_{\{#1\}}}}

 \usepackage{cleveref}

\begin{document}

\maketitle
{\def\thefootnote{}
\footnotetext{E-mail:
\texttt{jixu@cs.columbia.edu},
\texttt{djhsu@cs.columbia.edu}}}

\begin{abstract}
We study least squares linear regression over $N$ uncorrelated Gaussian features that are selected in order of decreasing variance. When the number of selected features $p$ is at most the sample size $n$, the estimator under consideration coincides with the principal component regression estimator; when $p>n$, the estimator is the least $\ell_2$ norm solution over the selected features. We give an average-case analysis of the out-of-sample prediction error as $p,n,N \to \infty$ with $p/N \to \alpha$ and $n/N \to \beta$, for some constants $\alpha \in [0,1]$ and $\beta \in (0,1)$. In this average-case setting, the prediction error exhibits a ``double descent'' shape as a function of $p$. We also establish conditions under which the minimum risk is achieved in the interpolating ($p>n$) regime.
 \end{abstract}

\section{Introduction}

In principal component regression (PCR), a linear model is fit to variables obtained using principal component analysis on the original covariates.
Suppose the data consists of $n$ i.i.d.\ observations $(\vx_1,y_1),\dotsc,(\vx_n,y_n)$ from $\bbR^N \times \bbR$.
Let $\vX := [ \vx_1 | \dotsb | \vx_n ]^\t$ be the $n \times N$ design matrix, $\vy := ( y_1, \dotsc, y_n )^\t$ be the $n$-dimensional vector of responses, and $\vSigma := \bbE[ \vx_1\vx_1^\t ] \in \bbR^{N \times N}$.
Assuming $\vSigma$ is known (as we do in this paper), the PCR fit is given by $\vV (\vX\vV)^+ \vy$, where $\vV \in \bbR^{N \times p}$ is the matrix of top $p$ (orthonormal) eigenvectors of $\vSigma$, and $\vA^+$ denotes the Moore-Penrose pseudo-inverse of $\vA$.
PCR notably addresses issues of multi-collinearity in under-determined ($n < N$) settings, while avoiding saturation effects suffered by other regression methods such as ridge regression~\citep{mathe2004saturation,bauer2007regularization,gerfo2008spectral}.

The critical parameter in PCR is the number of components $p$ to include in the regression.
Nearly all previous analyses of variable selection have restricted attention to the $p < n$ regime~\citep[e.g.,][]{breiman1983many}.
This restriction may seem benign, as conventional wisdom suggests that choosing $p > n$ leads to over-fitting.
This paper aims to challenge this conventional wisdom in a particular setting for PCR.

We study the prediction error of the PCR fit for all values of $p$ in the under-determined regime.
We assume the $\vx_i$ are Gaussian and conduct an ``average-case'' analysis, where the ``true'' coefficient vector is randomly chosen from an isotropic prior distribution.
Thus, all of the original variables in $\vx_i$ are relevant but weak in terms of predicting the response.
When the eigenvalues of $\vSigma$ exhibit some decay, one expects diminishing returns as $p$ increases.
It is often suggested to find a value of $p$ that balances bias and variance, and such a value of $p$ can be found in the $p < n$ regime.

However, we show that when $p > n$, the prediction error can again be decreasing with $p$.
This phenomenon---the second descent of the so-called ``double descent'' risk curve~\citep{belkin2018reconciling}---has been observed in a number of scenarios and for many different machine learning models (where $p$ is regarded as a nominal number of model parameters)~\citep{belkin2018reconciling,spigler2018jamming,belkin2019two,hastie2019surprises,muthukumar2019harmless}.
In these previous studies, the limiting risk as $p \to \infty$ was often (but not always) observed to be lower than the best risk achieved in the $p < n$ regime.
We prove that this phenomenon occurs with PCR in our data model: the lowest prediction error is achieved at some $p>n$, rather than any $p<n$.

\paragraph{Our data model.}
Our data $(\vx_1,y_1),\dotsc,(\vx_n,y_n)$ are assumed to be i.i.d.~with $\vx_i \sim \Normal(\v0,\vSigma)$, and
\[ y_i = \vx_i^\t \vtheta + w_i . \]
Here, $w_1,\dotsc,w_n$ are i.i.d.~$\Normal(0,\sigma^2)$ noise variables, and $\vtheta \in \bbR^N$ is the true coefficient vector.
We assume, without loss of generality, that $\vSigma$ is diagonal.
In fact, we shall take $\vSigma := \operatorname{diag}(\lambda_1,\dotsc,\lambda_N)$ with distinct positive eigenvalues $\lambda_1 > \dotsb > \lambda_N > 0$.
The prediction (squared) error of $\vtheta' \in \bbR^N$ is $\bbE_{\vx,y}[ (y - \vx^\t \vtheta')^2 ]$, where $(\vx,y)$ is an independent copy of $(\vx_1,y_1)$.

\emph{Some notation.}
For a vector $\vv \in \bbR^N$, let $\vv_P \in \bbR^p$ denote the sub-vector of the first $p$ entries of $\vv$, and let $\vv_{P^c} \in \bbR^{N-p}$ denote the sub-vector of the last $N-p$ entries.
Similarly, for a matrix $\vM \in \bbR^{n \times N}$, let $\vM_P \in \bbR^{n \times p}$ denote the sub-matrix of the first $p$ columns of $\vM$, and let $\vM_{P^c} \in \bbR^{n \times (N-p)}$ denote the sub-matrix of the last $N-p$ columns.

Recall that PCR selects components in order of decreasing $\lambda_j$.
So, using the notation from above, the PCR estimator $\hat\vtheta$ for $\vtheta$ is defined by
\begin{equation}
  \hat\vtheta_P
  \ :=
  \begin{cases}
    (\vX_P^\t \vX_P)^{-1} \vX_P^\t \vy & \text{if $p \leq n$} , \\
    \vX_P^\t (\vX_P \vX_P^\t)^{-1} \vy & \text{if $p > n$} ;
  \end{cases}
  \qquad
  \hat\vtheta_{P^c}
  \ := \
  \v0 .
  \label{eq:hatvtheta}
\end{equation}
(Recall that $\vX := [ \vx_1 | \dotsb | \vx_n ]^\t$ and $\vy := ( y_1, \dotsc, y_n )^\t$; also, the matrices being inverted above are, indeed, invertible with probability $1$.)
The prediction error of the PCR estimate $\hat\vtheta$ is denoted by 
\[ \error \ := \ \bbE_{\vx,y}[ (y - \vx^\t \hat\vtheta)^2 ] . \]

Observe that the (squared) correlation between the response and the $j$th variable is proportional to $\lambda_j \theta_j^2$, but PCR selects variables only on the basis of the $\lambda_j$.
So, for a worst-case $\vtheta$, PCR may be unlucky and end up selecting the $p$ least correlated variables.
To avoid this worst-case scenario, we consider an ``average-case'' analysis, where the true coefficient vector $\vtheta$ is independently drawn from an isotropic prior distribution:
\BE \bbE_{\vtheta}[ \vtheta ] = \v0 , \quad \bbE_{\vtheta}[ \vtheta\vtheta^\t ] = \vI . \label{eq:thetaass}\EE
We will study the random quantity $\bbE_{\vw,\vtheta}[ \error ]$, where the expectation is conditional on the design matrix $\vX$, but averages over the observation noise $\vw = (w_1,\dotsc,w_n)$ and random choice of $\vtheta$.

Our analysis uses high-dimensional asymptotic considerations to study the under-determined ($n < N$) regression problem, letting $p, n, N \to \infty$ with $p/N \to \ratiop$ and $n/N \to \ration$ for some fixed constants $\ratiop \in [0,1]$ and $\ration \in (0,1)$.
We are primarily interested in the limiting value of $\bbE_{\vw,\vtheta}[ \error ]$, which is the \emph{asymptotic risk}.

\paragraph{Our results.}

In \Cref{sec:special}, we give an exact expression for the asymptotic risk in the case where the eigenvalues of $\vSigma$ exhibit polynomial decay, namely $\lambda_j = j^{-\kappa}$ for a fixed constant $\kappa>0$.
Our expression covers both the $p < n$ and $p > n$ regimes, and we find that the smallest asymptotic risk can be achieved with $p > n$ (or equivalently, $\ratiop > \ration$) in noiseless settings.
In noisy settings, the comparison of the $p < n$ and $p > n$ regimes depends crucially on the exponent $\kappa$.

In \Cref{sec:general}, we relax the condition on the eigenvalues, and instead just assume that the empirical distribution of the $c_N\lambda_j$, for some suitable sequence $(c_N)_{N\geq1}$, has a ``nice'' limiting distribution.
We obtain results similar to those in \Cref{sec:special} using a slightly different variable selection rule.

Our analyses permit a $1-o(1)$ fraction of $\lambda_j$'s to converge to zero as $p,n,N\to\infty$. (In particular, the $c_N$ may go to infinity.) This makes our analysis technically non-trivial and more generally applicable.

\paragraph{Related works.}

Strategies for choosing the optimal value of $p$ in PCR (e.g., cross validation, variance inflation factors) are typically only studied in the $p < n$ regime~\citep{jolliffe2011principal}.
For instance, the exact risk of PCR as a function of $p$ for Gaussian designs can be extracted from the analysis of \citet{breiman1983many}, but only for the $p < n$ regime.

The high-dimensional analyses of ridge regression by \citet{dicker2016ridge,dobriban2018high,hastie2019surprises} are closely related to our work.
Indeed, for fixed $p$, the PCR estimator (or ``ridgeless'' estimator) is obtained by taking the ridge regularization parameter to zero.
These analyses extend beyond the Gaussian design setting that we consider, but are restricted to cases where either all eigenvalues of $\vSigma$ remain bounded below by an absolute constant as $N\to\infty$, or where the ridge regularization parameter is held at some positive constant.

The ``double descent'' phenomenon was observed by several researchers \citep[e.g.,][]{belkin2018reconciling,spigler2018jamming,muthukumar2019harmless,hastie2019surprises} for a variety of machine learning models such as neural networks and ensemble methods.
\citet{muthukumar2019harmless,belkin2019two,hastie2019surprises} provide statistical explanations for this phenomenon by studying the behavior of the minimum $\ell_2$ norm linear fit with $p > n$.
The analysis of \citet{muthukumar2019harmless} restricts attention to correctly-specified linear models (i.e., $p=N$ in our notation) and shows some potential benefits of the $p > n$ regime.
A related analysis of estimation variance was carried out by \citet{neal2018modern}.
The analysis of \citet{belkin2019two} studies an isotropic Gaussian design that is otherwise similar to our setup, as well as a Fourier design that is related to the random Fourier features of \citet{rahimi2008random}.
The analyses of \citet{hastie2019surprises} look at more general and non-isotropic designs (and, in fact, certain non-linear models related to neural networks!), but as mentioned before, they assume the eigenvalues of $\vSigma$ are bounded away from zero.
While their ``misspecified'' setting appears to be similar to our setup, we note that varying their $p/n$ parameter (which they call $\gamma$) changes the statistical problem under consideration.
In contrast, our analysis looks at the effect of choosing different $p$ on the same statistical problem, and thus is able to shed light on the number of variables one should use in principal component regression.

\paragraph{Notations for asymptotics.}
For any two random quantities $X$ and $Y$, we use the notation $X\ipt Y$ to mean that $X=Y+\op(Y)$ as $n,p,N\rightarrow \infty$.
Similarly, for any two non-random quantities $X$ and $Y$, we use the notation $X\rightarrow Y$ to mean that $X=Y+o(Y)$ as $n,p,N\rightarrow \infty$. Finally, we say that $X>Y$ holds in probability if $\Pr(X>Y)\rightarrow 1$ as $n,p,N\rightarrow \infty$.

\section{Analysis under polynomial eigenvalue decay}
\label{sec:special}

In this section, we analyze the asymptotic risk of PCR under the following assumptions:
\begin{itemize}
  \item[\A1] There exists a constant $\kappa>0$ such that $\lambda_j = j^{-\kappa}$ for all $j = 1,\dotsc,N$.
  \item[\A2] There exist constants $\ratiop \in \intcc{0,1}$ and $\ration \in \intoo{0,1}$ such that $p/N \to \ratiop$ and $n/N \to \ration$ as $p,n,N \to \infty$.
\end{itemize}
Assumption \A1 implies that the eigenvalues of $\vSigma$ decay to zero at a polynomial rate, while Assumption \A2 is a standard scaling for high-dimensional asymptotic analysis.

We also assume in this section that there is no observation noise, i.e., $\var(w_i) = \sigma^2 = 0$.
In the noiseless setting, the asymptotic risk is the limiting value of $\bbE_{\vtheta}[ \error ]$.
Results for the noisy setting are stated in \Cref{app:noise}.

\subsection{Main results}
\label{sec:special-results}

Our first \namecref{thm:p<n} provides characterizes the asymptotic risk when $\ratiop < \ration$.
Define the functions $h_{\kappa}$ and $\mathcal{R}_{\kappa}$ on $\intoo{0,\ration}$:\\
~\vspace{-15pt}
\begin{align}
  h_{\kappa}(\ratiop)
  & \ := \
  \frac{\ration }{\ratiop }-\int_{\ratiop }^1t^{\kappa-2}\dif t -1 ,
  \quad \text{for all $\ratiop < \ration$} ;
  \label{eq:h}
  \\
  \mathcal{R}_{\kappa}(\ratiop)
  & \ := \ N^{1-\kappa}\int_{\ratiop}^1t^{-\kappa}\dif t\cdot \frac{\ration }{\ration -\ratiop } ,
  \quad \text{for all $\ratiop < \ration$} .
  \label{eq:ipt}
\end{align}

\begin{theorem}\label{thm:p<n}
  Assume \A1 with constant $\kappa$; \A2 with constants $\ratiop$ and $\ration$; $\sigma^2 = 0$; and $\ratiop < \ration$.
  Then
  \begin{equation*}
    \bbE_{\vtheta}[\error] \ \ipt \ \mathcal{R}_{\kappa}(\ratiop) .
  \end{equation*}
  Furthermore, the equation $h_{\kappa}(\ratiop) = 0$ has a unique solution $\ratiop^*$ over the interval $\intoo{0,\ration}$, and
  $\mathcal{R}_{\kappa}(\ratiop)$ is decreasing on $\ratiop\in \intco{0,\ratiop^*}$ and increasing on $\ratiop\in (\ratiop^*,\ration)$.
  Finally,
  \begin{equation}
    \mathcal{R}_{\kappa}(\ratiop^*)
    \ = \
    \min_{0 \leq \ratiop < \ration }\ \mathcal{R}_{\kappa}(\ratiop)
    \ = \
    N^{1-\kappa}\frac{\ration }{(\ratiop^*)^{\kappa}}
    .\label{eq:minp<n}
  \end{equation}
\end{theorem}

The proof of \Cref{thm:p<n} is sketched in \Cref{sec:proofp<n}, with some details left to \Cref{app:p<n}.
\Cref{thm:p<n} supports the well-known intuition that the risk curve is ``U-shaped'' in the $p < n$ regime.
Our next \namecref{thm:p>n}, however, shows a very different behavior when $\ratiop > \ration$.

Formally define $m_{\kappa}(z)$ for $z \leq 0$ to be the smallest positive solution to the equation
\begin{equation}
	-z \ = \ \frac{1}{m_{\kappa}(z)} -\frac{1}{\ration } \int_{\ratiop ^{-\kappa}}^{\infty} \frac{1}{\kappa t^{1/\kappa}(1+t\cdot m_{\kappa}(z))}\dif t , \label{eq:mz}
\end{equation}
and let $m_{\kappa}'(\cdot)$ denote the derivative of $m_{\kappa}(\cdot)$.
Also define the function $\mathcal{R}_{\kappa}$ on $\intoc{\ration,1}$:
\begin{equation}
  \mathcal{R}_{\kappa}(\ratiop) \ := N^{1-\kappa}\left(\frac{\ration }{m_{\kappa}(0)}+\int_{\ratiop}^{1}t^{-\kappa}\dif t \cdot \frac{m_{\kappa}'(0)}{m_{\kappa}(0)^2}\right) ,
  \quad \text{for all $\ratiop > \ration$} .
  \label{eq:iptp>n}
\end{equation}

\begin{theorem} \label{thm:p>n}
  Assume \A1 with constant $\kappa$; \A2 with constants $\ratiop$ and $\ration$; $\sigma^2 = 0$; and $\ratiop > \ration$.
  The function $m_{\kappa}$ and its derivative $m_{\kappa}'$ are well-defined and positive at $z=0$ (and hence $\mathcal{R}_{\kappa}(\ratiop)$ is well-defined for all $\ratiop>\ration$).
  Moreover,
  \begin{equation*}
    \bbE_{\vtheta}[\error] \ \Ipt \ \mathcal{R}_{\kappa}(\ratiop) .
  \end{equation*}
\end{theorem}

The proof of \Cref{thm:p>n} is sketched in \Cref{sec:proofp>n}, with some details left to \Cref{app:p>n}.

We plot the asymptotic risk function $\mathcal{R}_{\kappa}$ in \Cref{fig:1} for two different values of $\kappa$, both with $\ration = 0.3$.
(In simulations, we find that $\mathbb{E}_{\vtheta}[\error]$ matches these curves very closely for sample sizes as small as $n=300$.)
For both values of $\kappa \in \{1,2\}$, we observe the striking ``double descent'' behavior as found in previous studies~\citep[e.g.,][]{belkin2018reconciling}.
Moreover, we see that the asymptotic risk at $\ratiop=1$ is smaller than the minimum asymptotic risk achieved at any $\ratiop < \ration$.
This, in fact, happens for all values of $\kappa>0$, as we claim in the next \namecref{thm:compare}.

\begin{theorem}\label{thm:compare}
   Assume \A1 with constant $\kappa$, \A2 with constants $\ratiop$ and $\ration$, $\sigma^2 = 0$. Let $\ratiop^*$ be the minimizer of $\mathcal{R}_{\kappa}$ over the interval $\intco{0,\ration}$.
  Then $\limsup_N\mathcal{R}_{\kappa}(1) / \mathcal{R}_{\kappa}(\ratiop^*) < 1$.
  Moreover,
  $\mathcal{R}_{\kappa}(\alpha) / \mathcal{R}_{\kappa}(1) \to \infty$ as $\ratiop \to \ration^{-}$.
\end{theorem}

The proof of \Cref{thm:compare} is given in \Cref{sec:compare}.
\Cref{thm:compare} shows that the asymptotic risk exhibits a second decrease somewhere in the $p>n$ regime when $N$ is sufficiently large, and moreover, that it is possible to find a value of $p$ in this $p>n$ regime to achieve a lower asymptotic risk than any $p<n$.

In the noisy setting (see \Cref{app:noise}), it is possible for the asymptotic risk to be dominated by the noise, in which case the minimum asymptotic risk is in fact achieved by $\ratiop = 0$ (i.e., $p = o(n)$). However, there exists a regime with $\sigma^2>0$ in which we have the same conclusion as in \Cref{thm:compare}.

\begin{figure}
  \centering
  \begin{tabular}{cc}
    \includegraphics[width=.5\textwidth]{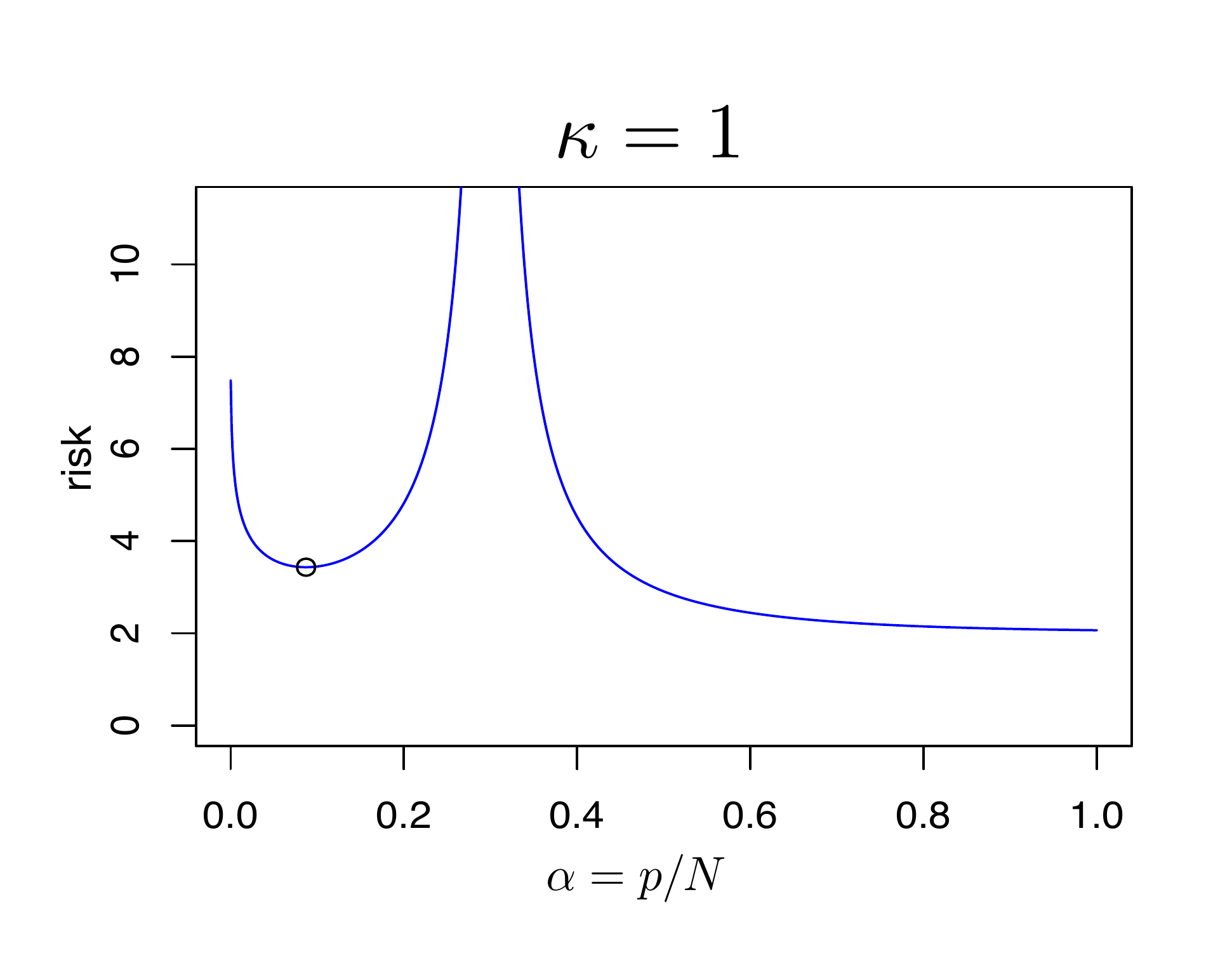} &
    \includegraphics[width=.5\textwidth]{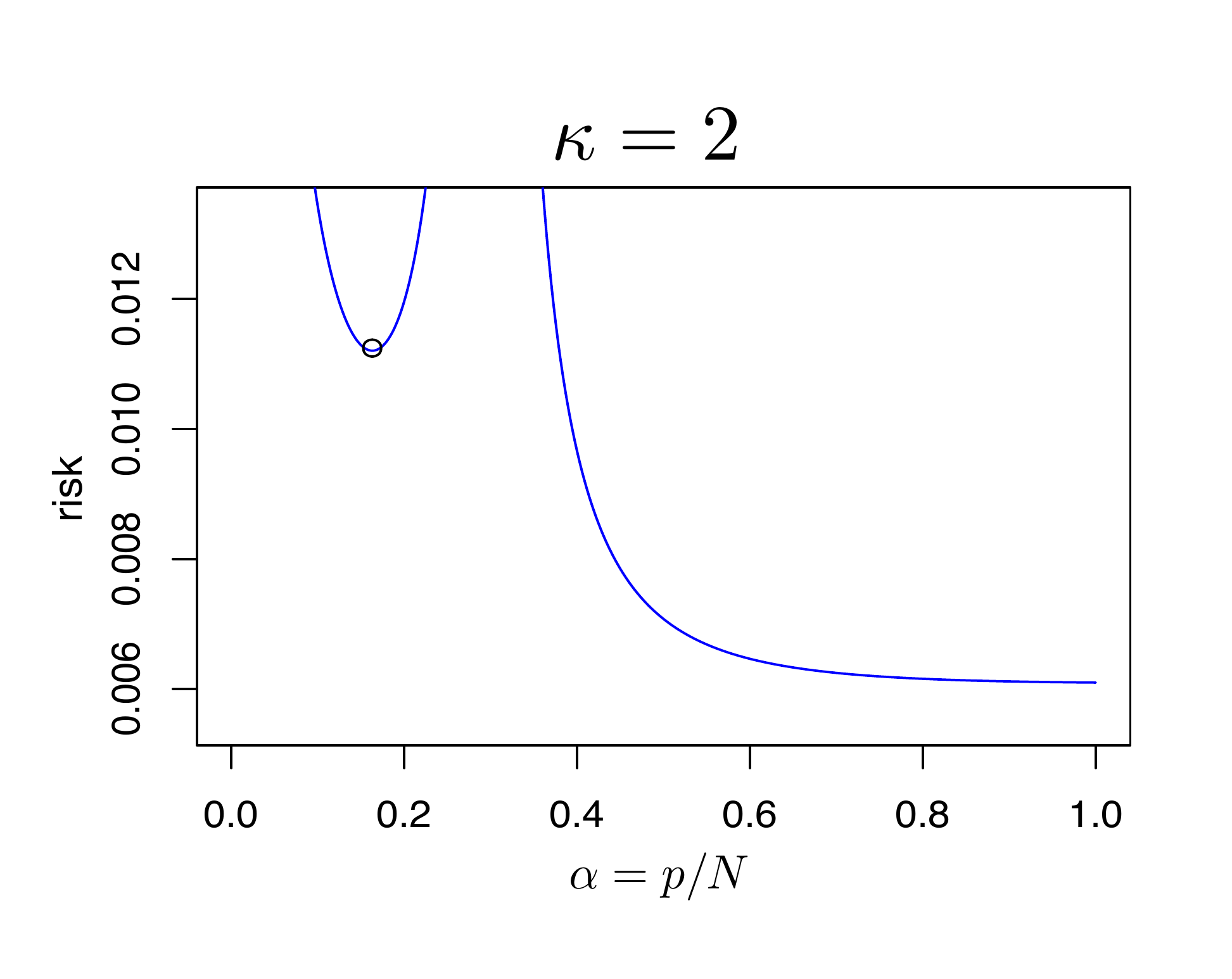}
  \end{tabular}
  \caption{The asymptotic risk function $\mathcal{R}_{\kappa}$ as a function of $\ratiop$ (with $n=300$, $N=1000$, $\ration = n/N = 0.3$ and $\kappa = 1, 2$ respectively).
    The location of $\ratiop^*$ from \Cref{thm:p<n} is marked with a black circle.
    In both cases, the asymptotic risk at $\ratiop=1$ is lower than the asymptotic risk at $\ratiop^*$.}
  \label{fig:1}
\end{figure}

\subsection{Proof sketch for \Cref{thm:p<n}}\label{sec:proofp<n}
We first show that $h_{\kappa}(\ratiop)=0$ has a unique solution on $\intoo{0,\ration}$.
Define $\tilde h_{\kappa}(\ratiop) := \ratiop^{1-\kappa}h_{\kappa}(\ratiop)$.
We shall show that $\tilde h_{\kappa}(\ratiop) = 0$ has a unique solution on $\intoo{0,\ration}$, which in turn immediately implies that $h_{\kappa}(\ratiop) = 0$ also has a unique solution on the same interval.
Observe that
\begin{equation}
	\frac{\dif \tilde h_{\kappa}(\ratiop)}{\dif \ratiop}  =
	\frac{-\kappa\ration +\kappa\ratiop}{\ratiop^{1+\kappa}}
  < 0
	.	\label{eq:hderiv}
\end{equation}
Hence, the function $\tilde h_{\kappa}(\ratiop)$ is strictly decreasing on $\ratiop \in \intoc{0,\ration}$.
Furthermore, we have
\begin{equation}
  \tilde h_{\kappa}(\ratiop) > 0 \ \ \text{as $\ratiop\rightarrow 0^{+}$} , \quad \text{and} \quad \tilde h_{\kappa}(\ratiop) < 0 \ \ \text{at $\ratiop=\ration$} .	\label{eq:p<nunique}
\end{equation}	
Because $\tilde h_{\kappa}$ is continuous, it follows that the equation $\tilde h_{\kappa}(\ratiop)=0$ has a unique solution on $\intoo{0,\ration}$.

We now prove $\bbE_{\vtheta}[\error] \ipt \mathcal{R}_{\kappa}(\ratiop)$.
Since the proof only requires standard techniques, we just sketch the main ideas in this section, and leave the full proof to \Cref{app:p<n}. First, since $\ratiop<\ration$, for large enough $N$, we have $p<n$. Then the prediction error is given by 
\BE
	\error
	&=&
  \bbE_{\vx,y}[(y-\vx^{\t}\hat{\vtheta})^2]
	\=
	\|\vSigmaP^{1/2}\left(\vXP^{\t}\vXP\right)^{-1}\vXP^{\t}\vXpC\vtheta_{P^c}\|^2+\|\vSigmapC^{1/2}\vtheta_{P^c}\|^2
	,	\n
\EE 
where $\vSigmaP\in\bbR^{p\times p}$ and $\vSigmapC\in \bbR^{(N-p)\times (N-p)}$ are two diagonal matrices whose diagonal elements are the first $p$ and last $N-p$ diagonal elements of $\vSigma$, respectively. By \eqref{eq:thetaass}, we have
\BE
  \bbE_{\vtheta}[\error]
	&=&
	\tr{\vXpC^{\t}\vXP\left(\vXP^{\t}\vXP\right)^{-1}\vSigmaP\left(\vXP^{\t}\vXP\right)^{-1}\vXP^{\t}\vXpC}+\tr{\vSigmapC}	
	.	\n
\EE
Note that $\vXpC$ is independent of $\vXP$, thus, given $\vXP$, the trace that includes $\vXpC$ is a sum of $N-p$ independent random variables.
Therefore, we have
\begin{eqnarray*}
  \bbE_{\vtheta}[\error]
	& \ipt &
	\tr{\vSigmapC} \cdot(\tr{(\vXP^{\t}\vXP)^{-1}\vSigmaP}+1)
		\\
  & = &
	\tr{\vSigmapC}\cdot(\tr{(\vbXP^{\t}\vbXP)^{-1}}+1)
  \\
	& \ipt &
	\tr{\vSigmapC} \frac{\ration}{\ration-\ratiop}
	,	\end{eqnarray*}
where $\vbXP := \vXP\vSigmaP^{-1/2}$ is a standard Gaussian matrix.
The first line above uses Markov's inequality to show that $\bbE_{\vtheta}[\error]$ converges in probability to $\bbE_{\vtheta,\vX_{P^c}}[\error]$.
The third line above uses Assumption $\A2$ and the fact that $\vbXP^{\t}\vbXP$ is a standard Wishart matrix $\mathcal{W}_p(\vI,n)$.
So, to prove \eqref{eq:ipt}, we just need to compute $\tr{\vSigmapC}$.
Note that $\int_{s}^{s+1}t^{-\kappa}\dif t<s^{-\kappa}<\int_{s-1}^st^{-\kappa}\dif t$. Hence, we have
\BE
	\int_{p+1}^{N}\frac{N^{\kappa}}{t^{\kappa}}\dif t \cdot \frac1N
	\ <\
	\frac{N^{\kappa}}{N}\sum_{i=p+1}^{N}\frac{1}{i^\kappa}
	\ = N^{\kappa-1}\tr{\vSigmapC}
	\ < \
	\int_{p}^N\frac{N^{\kappa}}{t^{\kappa}}\dif t \cdot \frac1N
	. \label{eq:tracep<nmain}
\EE
Therefore, we have $\tr{\vSigmapC}\rightarrow N^{1-\kappa}\int_{\ratiop}^1t^{-\kappa}\dif t$ as $p\rightarrow \infty$, and thus we have
$\bbE_{\vtheta}[\error] \ipt \mathcal{R}_{\kappa}(\ratiop)$.

Finally, to prove \eqref{eq:minp<n}, we analyze the shape of $\mathcal{R}_{\kappa}(\ratiop)$ to find its minimum value over $\ratiop<\ration$.
We take the derivative of $g_{\kappa}(\ratiop):=N^{\kappa-1}\mathcal{R}_{\kappa}(\ratiop)$:
\begin{equation}
	\frac{\dif g_{\kappa}(\ratiop)}{\dif \ratiop}
  \=
	\ration\cdot\frac{\ratiop^{1-\kappa}-\ration\ratiop^{-\kappa}+\int_\ratiop^1t^{-\kappa}\dif t}{(\ration -\ratiop)^2}
	\= \frac{-\ration\ratiop^{1-\kappa} \cdot h_{\kappa}(\ratiop)}{(\ration -\ratiop)^2}
	.\label{eq:derivp<n}
\end{equation}
Using
\eqref{eq:hderiv} and \eqref{eq:p<nunique}, we deduce that $\mathcal{R}_{\kappa}(\ratiop)$ first decreases and then increases as a function of $\ratiop$ in the interval $\intoo{0,\ration}$.
Therefore, the minimum risk is achieved at the unique solution $\ratiop^*$ of the equation $h_{\kappa}(\ratiop)=0$ over the interval $\intoo{0,\ration}$. Equation \eqref{eq:derivp<n} also implies $\int_{\ratiop^*}^1t^{-\kappa}\dif t = (\ration-\ratiop^*)(\ratiop^*)^{-\kappa}$. Hence, the minimum risk is given by
\begin{equation*}
	\min_{\ratiop<\ration }\ \mathcal{R}_{\kappa}(\ratiop)
	\=
  N^{1-\kappa}\frac{\ration }{\ration -\ratiop^*}\int_{\ratiop^*}^1t^{-\kappa}\dif t
	\=
	N^{1-\kappa}\frac{\ration }{(\ratiop^*)^{\kappa}}.
\end{equation*}

\subsection{Proof sketch for \Cref{thm:p>n}} \label{sec:proofp>n}

We first show that $m_{\kappa}(0)$ is well-defined.
Consider the RHS expression from \Cref{eq:mz} evaluated at $z=0$; by a change-of-variable in the integral, we have
\begin{equation}
  \frac1m - \frac1{\ration} \int_{\ratiop^{-\kappa}}^\infty \frac{\kappa^{-1} t^{-1/\kappa}}{1 + t \cdot m} \dif t
  =
  \frac{1}{\ration \ratiop m^{1-1/\kappa}}\left(\frac{\ration}{m^{1/\kappa}/\ratiop}-\ratiop \int_{m^{1/\kappa}/\ratiop}^{\infty} \frac{t^{\kappa-2}}{1+t^{\kappa}}\dif t\right),
  \label{eq:transformeq}
\end{equation}
where $m=m_{\kappa}(0)$. So, we just need to show that $q_{\kappa}(s,\ratiop)=0$ has a unique solution $s^*_{\kappa}$ for $s$ over the positive real line, where $q_{\kappa}(s,\ratiop)$ is defined by
\BE
q_{\kappa}(s,\ratiop) \ := \ \frac{\ration }{s}-\ratiop\int_s^{\infty}\frac{t^{\kappa-2}}{1+t^{\kappa}}\dif t.\label{eq:q}
\EE
(This makes $m_\kappa(0)$ well-defined, via the equation $s^*_{\kappa} = m_\kappa(0)^{1/\kappa}/\ratiop$, and also verifies its positivity.)

The derivative of $q_{\kappa}(s,\ratiop)$ with respect to $s$ is
\BE
	\frac{\partial q_{\kappa}(s,\ratiop)}{\partial s}
	&=&
	\frac{(\ratiop-\ration )s^{\kappa}-\ration }{s^2(1+s^{\kappa})}
	.	\label{eq:qpartial}
\EE
Hence, since $\ratiop>\ration $, we know the function $q_{\kappa}(s,\ratiop)$ is strictly decreasing on $s\in (0,(\frac{\ration }{\ratiop-\ration })^{1/\kappa}]$ and strictly increasing on $s\in [(\frac{\ration }{\ratiop-\ration })^{1/\kappa}, \infty)$. Furthermore, $q_{\kappa}(s,\ratiop)\rightarrow \infty$ as $s\rightarrow 0$ and $q_{\kappa}(s,\ratiop)\rightarrow 0$ as $s\rightarrow \infty$. Hence, by the continuity of $s \mapsto q_\kappa(s,\ratiop)$, we conclude that $q_{\kappa}(s,\ratiop)=0$ has a unique solution $s^*_{\kappa}$.

Using the chain rule, we can also show that $m_{\kappa}'(0)$ is well-defined, and that its value is given by
\[
  m_{\kappa}'(0)\=\kappa\ration m^2_{\kappa}(0) \cdot (1+(s^*_{\kappa})^{\kappa})/\left(\ration+(\ration-\ratiop)(s^*_{\kappa})^{\kappa}\right) > 0.
\]
We leave the details to \Cref{app:monoto}.

Our next goal is to prove
   $\bbE_{\vtheta}[\error] \Ipt \mathcal{R}_{\kappa}(\ratiop)$.
Since $\ratiop>\ration $, we have $p>n$ for large enough $N$.
In this case,
\BE
	\error
	&=&
  \bbE_{\vx,y}[(y-\vx^{\t}\hat{\vtheta})^2]
	\=
  \bbE_{\vx,y}[(\vx_P^{\t}(\hat{\vtheta}_P-\vtheta_P)-\vx_{P^c}^{\t}\vtheta_{P^c})^2]
		\n\\
	&=&
	\|\vSigmaP^{1/2}((\vPi_{\vXP}-\vI)\vtheta_P+\vXP^{\t}(\vXP\vXP^{\t})^{-1}\vXpC\vtheta_{P^c})\|^2+\|\vSigmapC^{1/2}\vtheta_{P^c}\|^2
	,	\n
\EE 
where $\vPi_{\vXP} := \vXP^{\t}\left(\vXP\vXP^{\t}\right)^{-1}\vXP$, and the diagonal matrices $\vSigmaP$ and $\vSigmapC$ are as defined in \Cref{sec:proofp<n}.
Hence, $\bbE_{\vtheta}[\error]$ is equal to
\begin{equation}
\underbrace{\tr{\vSigmaP(\vI-\vPi_{\vXP})}}_{\text{part 1}}
  + \underbrace{\tr{\vXpC^{\t}(\vXP\vXP^{\t})^{-1}\vXP\vSigmaP\vXP^{\t}(\vXP\vXP^{\t})^{-1}\vXpC} + \tr{\vSigmapC}}_{\text{part 2}} .
	\label{eq:bbEwthetaerrorp>n}
\end{equation} 
We claim that
\BE
	\text{part 1}
	& \ipt &
	\frac{N^{1-\kappa}\ration }{m_{\kappa}(0)}, \quad \text{and}\quad	\text{part 2} 
	\ \ipt \ 
	N^{1-\kappa}\cdot \frac{m_{\kappa}'(0)}{m^2_{\kappa}(0)} \cdot \int_{\ratiop}^1t^{\kappa-2}\dif t+\op(N^{1-\kappa})
	;	\label{eq:p>npart12}
\EE
together, they complete the proof that
   $\bbE_{\vtheta}[\error] \Ipt \mathcal{R}_{\kappa}(\ratiop)$.
Rigorous proofs of the claims in \eqref{eq:p>npart12} are presented in \Cref{app:part1} and \Cref{app:part2}; here, we give a heuristic argument that conveys the main idea.
For part 1, let $\vtSigmaP=N^{\kappa}\vSigmaP$ and $\vtXP=N^{\kappa/2}\vXP$.
This scaling ensures that the empirical eigenvalue distribution of $\vtSigmaP$ has a limiting distribution with probability density
\[ f_{\kappa}(s) =\frac{1}{\kappa\ratiop}s^{-1-1/\kappa}\cdot \ind{s \in [\ratiop^{-\kappa},\infty)} \]
(\Cref{lem:dH} in \Cref{app:part1}).
Also, under this scaling, we have
\BE
	\lefteqn{\tr{\vSigmaP\left(\vI-\vPi_{\vXP}\right)}
	\=
	\lim_{\mu\rightarrow 0}\frac{n}{N^{\kappa}}\left(\frac{1}{n}\tr{\vtSigmaP}-\frac{1}{n}\tr{\vtSigmaP (\vtXP^{\t}\vtXP+\mu n\vI)^{-1}\vtXP^{\t}\vtXP}\right)}
		\n\\
	&=&
	\lim_{\mu\rightarrow 0}\frac{n}{N^{\kappa}}\cdot\frac{\mu}{n}\trace{\vtSigmaP \left(\frac{1}{n}\vtXP^{\t}\vtXP+\mu \vI\right)^{-1}}
	\=
	\lim_{\mu\rightarrow 0}\frac{n}{N^{\kappa}}\cdot\frac{\mu}{n}\tr{\vtSigmaP \vtS_n}
	,	\label{eq:p>neq_20}
\EE
where $\vtS_n := (n^{-1}\vtXP^{\t}\vtXP+\mu \vI)^{-1}$.
As long as the empirical eigenvalue distribution of $\vtSigmaP$ has a limiting distribution with bounded support, we have
\BE
	\forall \mu>0, \quad \mu\cdot\frac{1}{n}\trace{\vtSigmaP \left(\frac{1}{n}\vtXP^{\t}\vtXP+\mu \vI\right)^{-1}}
	\ \ipt \ 
	\frac{1}{m_{\kappa}(-\mu)}
	,	\label{eq:heuristics0}
\EE
where $m_{\kappa}(z)$ is, in fact, the Stieltjes transform of the limiting empirical eigenvalue distribution of $n^{-1}\vtXP\vtXP^{\t}$ (\Cref{lem:St} in \Cref{app:part1}); this follows from results of
\citet{dobriban2018high}, which in turn are derived from the results of \citet{ledoit2011eigenvectors}.
Assume we can exchange the two limits $\mu\rightarrow 0^+$ and $N\rightarrow \infty$, and also that \eqref{eq:heuristics0} still holds for $f_{\kappa}(s)$ which has unbounded support.
Then, from \eqref{eq:p>neq_20}, we conclude
\BE
	 \text{part 1} \= \tr{\vSigmaP\left(\vI-\vPi_{\vXP}\right)} \ \ipt \ \frac{N^{1-\kappa}\ration }{m_{\kappa}(0)}
	.	\n
\EE

For part 2, note that $\vXpC$ is independent of $\vXP$.
Thus, conditional on $\vXP$, part 2 is a sum of $N-p$ independent random variables. Therefore, using Markov inequality, we can show that 
\BE
\text{part 2} &\ipt& \bbE_{\vXpC}[\text{part 2}] \=\trace{\vSigmapC}\cdot\left(\trace{\vSigmaP\vXP^{\t}\left(\vXP\vXP^{\t}\right)^{-2}\vXP}+1\right)
		\n\\
	&=&
	\trace{\vSigmapC}\cdot\left(\lim_{\mu\rightarrow 0}\trace{\vtSigmaP\left(\vtXP^{\t}\vtXP+\mu n\vI\right)^{-1}\vtXP^{\t}\vtXP\left(\vtXP^{\t}\vtXP+\mu n\vI\right)^{-1}}+1\right)
		\n\\
	&=&
	\trace{\vSigmapC}\cdot\left(\lim_{\mu\rightarrow 0}\frac{1}{n}\trace{\vtSigmaP\vtS_n}-\frac{\mu}{n}\trace{\vtSigmaP\vtS_n^2}+1\right)
	.\label{eq:part23_eq1}
\EE
Again, if we ignore the fact that the support of $f_{\kappa}(s)$ is unbounded and assume the limits of $\mu\rightarrow 0$ and $N\rightarrow \infty$ can be exchanged, then by Lemma 7.4 of \citet{dobriban2018high}, we have 
\BE
	\text{part 2} \ \ipt\ \trace{\vSigmapC}\cdot\left(\lim_{\mu\rightarrow 0}\frac{1}{n}\trace{\vtSigmaP\vtS_n}-\frac{\mu}{n}\trace{\vtSigmaP\vtS_n^2}+1\right)\ipt  \trace{\vSigmapC}\cdot\frac{m_{\kappa}'(0)}{m_{\kappa}^2(0)}.	\label{eq:part23_eq2}
\EE
A straightforward analysis of
$\tr{\vSigmapC}$ (as in \eqref{eq:tracep<nmain}) completes the analysis of part 2 of \eqref{eq:p>npart12}.

\begin{remark}
  Although \Cref{thm:p>n} should intuitively hold given the results of \citet{dobriban2018high}, a careful and more involved argument is needed to deal with the facts that $\|\vtSigmaP\|_2\rightarrow \infty$ (since $\|\vSigmaP^{-1}\|_2\rightarrow \infty$) and $\mu\rightarrow 0$.
  For example, standard techniques only imply $\frac{\mu}{n}\tr{\vtSigmaP\vtS_n}=\Op(N^{\kappa})$.
  However, we need the stronger bound $\frac{\mu}{n}\tr{\vtSigmaP\vtS_n} = \Op(1)$ (e.g., \Cref{hardpart}).
\end{remark}

\subsection{Proof of Theorem \ref{thm:compare}}\label{sec:compare}

Comparing the expression for $\mathcal{R}_{\kappa}(\ratiop)$ in \eqref{eq:iptp>n} at $\ratiop=1$ to the expression for $\mathcal{R}_{\kappa}(\ratiop^*)$ in \eqref{eq:minp<n}, we see that it suffices to prove $m_{\kappa}(0)^{1/\kappa}>\ratiop^*$.
Recall that in \Cref{sec:proofp>n}, we have proved $s^*_{\kappa} := m_{\kappa}(0)^{1/\kappa}$ is the unique solution of the equation $q_{\kappa}(s,1)=0$.
Furthermore, using the expression for the derivative of $q_{\kappa}(s,1)$ with respect to $s$ in \eqref{eq:qpartial}, we know that $q(s,1) > 0 \Rightarrow s < s^*_{\kappa}$.
Thus, we only need to show $q_{\kappa}(\ratiop^*,1)>0=h_{\kappa}(\ratiop^*)$, where the equality is due to the definition of $\ratiop^*$ in \Cref{thm:p<n}.
Note that by the definitions of the functions $q_{\kappa}$ and $h_{\kappa}$ in \eqref{eq:h} and \eqref{eq:q}, we have
\BE
	h_{\kappa}(s)
	&=&
	\frac{\ration }{s}-\int_{s}^{1}t^{\kappa-2}\dif t-1 
	\=
	q_{\kappa}(s,1)+\int_s^{\infty}\frac{t^{\kappa-2}}{(1+t^{\kappa})}\dif t-\int_s^1t^{\kappa-2}\dif t-1
	.\n
\EE 
Furthermore, $h_{\kappa}(s) - q_{\kappa}(s,1)$ is increasing in $s$:
\BE
	\frac{\dif~ (h_{\kappa}(s)-q_{\kappa}(s,1))}{\dif s}
	&=&
	-\frac{s^{\kappa-2}}{(1+s^{\kappa})}+s^{\kappa-2}
	\=
	\frac{s^{2\kappa-2}}{1+s^{\kappa}}
	\ >\ 0.\n
\EE
Hence, for all for all $s\in \intoc{0,1}$, we have
\BE
	h_{\kappa}(s)-q_{\kappa}(s,1)
	&\leq&
	h_{\kappa}(1)-q_{\kappa}(1,1)
	\=
	\int_1^{\infty}\frac{t^{\kappa-2}}{(1+t^{\kappa})}\dif t-1
		\n\\
	&=&
	\int_1^{\infty}\frac{t^{\kappa-2}}{(1+t^{\kappa})}\dif t-\int_1^{\infty} \frac{1}{t^2}\dif t
	\=
	-\int_1^{\infty}\frac{1}{t^2(1+t^{\kappa})}\dif t \ <\ 0.\n
\EE
Since $\ratiop^*<\ration <1$, we have $0=h_{\kappa}(\ratiop^*)<q_{\kappa}(\ratiop^*,1)$, and thus we have $s^*_{\kappa}>\ratiop^*$. 

By inspection of the expression for $\mathcal{R}_{\kappa}(\ratiop)$ in \eqref{eq:ipt}, it is also clear that $\mathcal{R}_{\kappa}(\ratiop)/\mathcal{R}_{\kappa}(1)\rightarrow \infty$ as $\ratiop\rightarrow \ration^{-}$.

\section{Analysis under general eigenvalue decay}\label{sec:general}

In this section, we extend the results from \Cref{sec:special} (with noise) to hold under a more general assumption on the eigenvalues of $\vSigma$.
To simplify calculations, we use a slightly different feature selection procedure that includes all components $j$ such that $\lambda_j \geq \nu_N$, so $p=\sum_{j=1}^N \ind{\lambda_j\geq\nu_N}$.

Instead of Assumptions $\A1$ and $\A2$, we assume the following:
\begin{itemize}
\item[$\B1$] $\|\vSigma\|_2 \leq C$ for some constant $C>0$. Also, there exists a positive sequence $(c_N)_{N\geq1}$ such that the empirical eigenvalue distribution of $c_N\vSigma$ converges as $N\to\infty$ to $F=(1-\delta)F_0+\delta F_1$, where $\delta\in (0,1]$, $F_0$ is a point mass of $0$, and $F_1$ has a continuous probability density $f$ supported on either $\intcc{\eta_1,\eta_2}$ or $\intco{\eta_1,\infty}$ for some constants $\eta_1,\eta_2>0$.

\item[$\B2$]
  There exist constants $\nu>0$ and $\ration \in \intoo{0,\delta}$ s.t.\
  $\nu_Nc_N \to \nu$
  and
  $n/N \to \ration$
  as $n,N\to\infty$.
\end{itemize}

The $c_N$ in Assumption $\B1$ generalizes the $N^{\kappa}$ scaling introduced in the proof of \Cref{thm:p>n}.
In fact, Assumption $\B1$ is more general than the eigenvalue assumptions made by \citet{dobriban2018high} and \citet{hastie2019surprises}: the eigenvalues of $\vSigma$ could decrease smoothly ($\delta=1$), or there could be a sudden drop between (say) $\lambda_j$ and $\lambda_{j+1}$ ($\delta<1$).
Since $p$ is now determined by $\nu$, whether $p<n$ or $p>n$ is now determined by whether $\nu>\nu_{b}$ or $\nu<\nu_{b}$, where $\nu_{b}> \eta_1$ is given by the equation $\delta\int_{\nu_{b}}^{\infty} f(t)\dif t = \ration$.
Finally, by Assumption $\B1$,
\BE
\frac{p}{N}\=\frac{1}{N}\sum_{j=1}^N \ind{c_N\lambda_j\geq \nu}\  \ast \ \delta \bbE_{s\sim f}[\ind{s\geq \nu}] \= \delta\int_{\nu}^{\infty} f(t)\dif t \ =: \ \ratiop(\nu), \quad \forall \nu>0. \label{eq:gamma}
\EE
For $\nu=0$, i.e., $\nu_N=o(1/c_{N})$, we choose $\nu_N$ be the $\delta N$ largest eigenvalues of $\vSigma$, then $\ratiop(\nu) = \delta$. Hence, combined with Assumption $\B2$, we have the same asymptotics considered in \Cref{sec:special}, except that $\ration$ is now restricted in $(0,\delta)$. This restriction on $\ration$ is required, otherwise both $c_N\vX^{\t}\vX$ and $c_N\vX\vX^{\t}$ are asymptotically singular.

The following \namecref{thm:general} generalizes the results in \Cref{sec:special} to hold under Assumptions $\B1$ and $\B2$.

\begin{theorem}\label{thm:general}
  Assume $\B1$ with sequence $(c_N)_{N\geq1}$ and constants $C$, $\delta$, $\eta_1$, and $\eta_2$; and $\B2$ with constants $\nu$ and $\ration$.
\begin{itemize}
\item[(i)] Assume $\nu \in (\nu_{b},\infty)$. Then
\BE
\bbE_{\vw,\vtheta}[\error] \ \ipt \ \left( \frac{N}{c_N}\cdot\delta\int_{\eta_1}^{\nu}tf(t)\dif t+\sigma^2\right)\frac{\ration}{\ration-\delta\int_{\nu}^{\infty} f(t)\dif t } \ =: \ \mathcal{R}_f(\nu,\sigma). \label{eq:generalp<nthm_1}
\EE
Define $h_f(\nu) := \nu\ration-\nu\delta\int_{\nu}^{\infty}f(t)\dif t-\delta\int_{\eta_1}^{\nu}tf(t)\dif t$.
If the equation $h_f(\nu)=0$ has a solution on $(\nu_{b},\infty)\bigcap \text{supp}(f)$, then the solution $\nu^*$ is unique, and
\begin{equation}
  \mathcal{R}_f(\nu^*,0) \=
	\min_{\nu \in (\nu_{b},\infty)} \mathcal{R}_f(\nu,0)
	\=
    \frac{N\ration}{c_N}\cdot \nu^*
	.\label{eq:generalminp<n1}
\end{equation}
Otherwise, \begin{equation}
	\inf_{\nu \in (\nu_{b},\infty)} \mathcal{R}_f(\nu,0)
  \=
  \lim_{\nu\to\infty} \mathcal{R}_f(\nu,0)
	\=
	\frac{N}{c_N} \delta\int_{\eta_1}^{\infty}tf(t)\dif t
	.\label{eq:generalminp<n2}
\end{equation}
\item[(ii)]
  Assume $\nu \in \intco{0,\nu_{b}}$.
  Define $q_f(s,\nu) := s\ration - s\delta\int_{\nu}^{\infty} \frac{tf(t)}{s+t}\dif t$.
  Then
\BE
  \bbE_{\vw,\vtheta}[\error]
	&\ipt&
	\frac{N\ration }{c_N} s^*_f+\ration\cdot\frac{\frac{N}{c_N}\delta\int_{\eta_1}^{\nu}tf(t)\dif t+\sigma^2}{\delta s^*_f\int_{\nu}^{\infty}\frac{tf(t)}{(s^*_{f}+t)^2}\dif t} \ =: \ \mathcal{R}_f(\nu,\sigma)
	,	\label{eq:generalriskp=N}
\EE
where $s^*_f$ is the unique solution of the equation $q_{f}(s,\nu)=0$.

\item[(iii)]Suppose $\sigma=0$. Let $\nu^*$ be the minimizer of $\mathcal{R}_{f}(\nu, 0)$ over the interval $(\nu_b,\infty]$ (including $\infty$). Let $\mathcal{R}_f(\eta_1,0)$ be the risk achieved at $\nu=\eta_1$. Then $\limsup_N\mathcal{R}_{f}(\eta_1,0) / \mathcal{R}_{f}(\nu^*, 0) < 1$.

\end{itemize}
\end{theorem}

The proof of this theorem is presented in Appendix \ref{app:general}.

\section{Discussion}

Our results confirm the emergence of the ``double descent'' risk curve in a natural setting with Gaussian design.
As in previous works~\citep[e.g.,][]{belkin2019two,hastie2019surprises,muthukumar2019harmless}, the shape emerges when there is a spike at the interpolation threshold ($p=n$), which is typically caused by a near-zero minimum eigenvalue of the empirical covariance matrix.

More importantly, however, our results shed light on when the minimum risk is achieved before or after the interpolation threshold in terms of the noise level and eigenvalues of the (population) covariance matrix.
For instance, when the eigenvalues decay very slowly or not at all ($\kappa<1$), a smaller risk is achieved after the interpolation threshold ($p>n$) than any point before ($p<n$).
On the other hand, when the eigenvalues decay more quickly ($\kappa>1$), a smaller risk is achieved in the $p>n$ regime only in the noiseless setting.
In general, the $p<n$ regime yields a smaller risk when the noise dominates the error due to model misspecification.
Providing a full characterization is an important direction for future research.

Finally, we point out that the PCR estimator we study is a non-standard ``oracle'' estimator because it generally requires knowledge of $\vSigma$.
Although it can be plausibly implemented in a semi-supervised setting (by estimating $\vSigma$ very accurately using unlabeled data), a full analysis that accounts for estimation errors in $\vSigma$, or of a more standard PCR estimator, remains open.
However, we note that the PCR estimator with $p=N$ can be implemented, and in our analysis, the dominance of the $p>n$ regime is always established at $p=N$.
We believe that this should be true for the standard PCR estimator as well.

\subsection*{Acknowledgments}
This research was supported by NSF CCF-1740833, a Sloan Research Fellowship, a Google Faculty Award, and a Cheung-Kong Graduate School of Business Fellowship.

\bibliographystyle{plainnat}
\begin{appendix}
\section{Proof of \Cref{thm:p<n}}\label{app:p<n}

The full proof for \Cref{thm:p<n} is presented in this section. Since $\ratiop<\ration$, we have $p<n$ hold for large enough $N$. Then, the least square estimate $\hat{\vtheta}_P$ is given by $\left(\vXP^{\t}\vXP\right)^{-1}\vXP^{\t}\vX\vtheta$ and the prediction error is given by
\BE
	\error
	&=&
  \bbE_{\vx,y}[(y-\vx^{\t}\hat{\vtheta})^2]
	\=
  \bbE_{\vx,y}[(\vx_P^{\t}(\vtheta_P-\hat{\vtheta}_P)+\vx_{P^c}^{\t}\vtheta_{P^c})^2]
		\n\\
	&=&
	\|\vSigmaP^{1/2}(\vXP^{\t}\vXP)^{-1}\vXP^{\t}\vXpC\vtheta_{P^c}\|^2+\|\vSigmapC^{1/2}\vtheta_{P^c}\|^2
	,	\n
\EE 
where $\vSigmaP\in\bbR^{p\times p}$ and $\vSigmapC\in \bbR^{(N-p)\times (N-p)}$ are the two diagonal matrices whose diagonal elements are the first $p$ and last $N-p$ diagonal elements of $\vSigma$ respectively. By our assumption on $\vtheta$, we have
\BE
\bbE_{\vtheta}[\error]
	&=&
	\tr{\vXpC^{\t}\vXP(\vXP^{\t}\vXP)^{-1}\vSigmaP(\vXP^{\t}\vXP)^{-1}\vXP^{\t}\vXpC}+\tr{\vSigmapC}	
.	\n
\EE
Our next step is to apply Markov inequality to show \eqref{eq:ipt}. Note that $\vXpC$ is independent of $\vXP$. Hence, the expectation of $\error$ given $\vXP$ is the following:
\BE
	\bbE[\error \mid \vXP]
	&=&
	\tr{\vSigmapC}\cdot(\tr{(\vXP^{\t}\vXP)^{-1}\vSigmaP}+1)
		\n\\
	&=&
	\tr{\vSigmapC}\cdot(\tr{(\vbXP^{\t}\vbXP)^{-1}}+1)
	,	\label{eq:p<nexpect}
\EE
where $\vbXP=\vXP\vSigmaP^{-\frac{1}{2}}$.
(The expectation only conditions on $\vXP$; in particular, it averages over $\vXpC$.)
Further, the variance of $\error$ given $\vXP$ is the following:
letting $\vz \sim \Normal(\v0,\vI)$,
\BE
  \lefteqn{\var(\error \mid \vXP)}
  \n
  \\
  & = &
	\tr{\vSigmapC^2}\var(\vz^{\t}\vXP(\vXP^{\t}\vXP)^{-1}\vSigmaP(\vXP^{\t}\vXP)^{-1}\vXP^{\t}\vz \mid \vXP)
		\n\\
	&\leq&2
	\tr{\vSigmapC^2}\|\vXP(\vXP^{\t}\vXP)^{-1}\vSigmaP(\vXP^{\t}\vXP)^{-1}\vXP^{\t}\|_{F}^2
		\n\\
	&=&
	2\tr{\vSigmapC^2}\tr{(\vXP^{\t}\vXP)^{-1}\vSigmaP(\vXP^{\t}\vXP)^{-1}\vSigmaP} \n \\
  & = &
	2\tr{\vSigmapC^2}\tr{(\vbXP^{\t}\vbXP)^{-2}}
	.	\n
\EE  
Hence, by Markov's inequality and the fact that $\tr{\vSigmapC^2}\leq \tr{\vSigmapC}^2$, we have
\BE
\bbE_{\vtheta}[\error]
	=
  \bbE[ \error|\vXP]\cdot \del{ 1+\Op\del{\tr{(\vbXP^{\t}\vbXP)^{-2}}^{1/2}\cdot \del{ \tr{(\vbXP^{\t}\vbXP)^{-1}}+1 }^{-1}} }
	. 	\label{eq:markov}
\EE
Our next step is to simplify \eqref{eq:markov}. Note that $\vbXP$ is a standard Gaussian matrix. Hence, when $\ratiop>0$, from (2.104) and (2.105) of \cite{tulino2004random}, we know
\BE
	\frac{n}{p}\trace{\left(\vbXP^{\t}\vbXP\right)^{-1}}
	\ \ast \ 
	\frac{\ration}{\ration -\ratiop}
	\qand	
	\frac{n^2}{p}\cdot \trace{\left(\vbXP^{\t}\vbXP\right)^{-2}}
	\ \ast \
	\frac{\ration^3}{(\ration -\ratiop)^3}
	.	\n
\EE
When $\ratiop=0$, i.e., $p=o(n)$, from (2.110) and (2.111) of \cite{tulino2004random}, we know 
\BE
	\frac{n}{p}\trace{\left(\vbXP^{\t}\vbXP\right)^{-1}}
	\ \ast \
	1	\qand	
	\frac{n^2}{p} \trace{\left(\vbXP^{\t}\vbXP\right)^{-2}}
	\ \ast \
	1.	\n
\EE
Therefore, with \eqref{eq:p<nexpect} and \eqref{eq:markov}, we have for all $\ratiop<\ration$,
\BE
  \bbE_{\vtheta}[\error]
	&=&
	\bbE[ \error|\vXP]\cdot \left(1+\Op\left(\sqrt{\frac{\ration \ratiop(\ration -\ratiop)^{-3}}{N\left(\frac{\ratiop}{\ration -\ratiop}+1\right)^2}}\right)\right) \n
  \\
  & = &
	 \bbE[ \error|\vXP]\cdot \left(1+\Op\left(\frac{1}{\sqrt{N}}\right)\right)
		\n\\
  & \ipt &
	\trace{\vSigmapC}\cdot\left(\trace{\left(\vbXP^{\t}\vbXP\right)^{-1}}+1\right)
	\ \ipt \
	\trace{\vSigmapC}\cdot\frac{\ration }{\ration -\ratiop}. \label{eq:ipta>0}
\EE
Our final step is to analyze $\trace{\vSigmapC}$.
Note that $\int_{s}^{s+1}t^{-\kappa}\dif t<\frac{1}{s^\kappa}<\int_{s-1}^st^{-\kappa}\dif t$.
Hence, we have
\BE
	\int_{p+1}^{N}\frac{N^{\kappa}}{t^{\kappa}}\dif t/N
	\ <\
	\frac{N^{\kappa}}{N}\sum_{i=p+1}^{N}\frac{1}{i^\kappa}
	\ = N^{\kappa-1}\trace{\vSigmapC}
	\ < \
	\int_{p}^N\frac{N^{\kappa}}{t^{\kappa}}\dif t/N
	. \label{eq:tracep<n2}
\EE
Therefore, we know $\trace{\vSigmapC}\rightarrow N^{1-\kappa}\int_{\ratiop}^1t^{-\kappa}\dif t$ as $p\rightarrow \infty$ and thus \eqref{eq:ipt} holds.

\section{Proof of \Cref{thm:p>n}}\label{app:p>n}

\subsection{Existence and positivity of $m_{\kappa}'(0)$}\label{app:monoto}

We already showed in \Cref{sec:proofp>n} that $m_{\kappa}(0)$ is well-defined.
We now show that $m_{\kappa}(z)$ is well-defined in a neighborhood of $z=0$, which we can then use to establish the existence and positivity of $m_{\kappa}'(0)$.
Note that, in fact, \Cref{lem:St} in \Cref{app:part1} shows that $m_{\kappa}(z)$ is the Stieltjes transform of a distribution, specifically the limiting distribution of the empirical eigenvalue distribution of $\vSigmaP$.
This \namecref{lem:St}, which is proved in \Cref{app:St}, establishes the existence of the Stieltjes transform for all $z\leq 0$.
Here, we just give the arguments needed to show the existence of $m_{\kappa}'(0)$.

Define 
\[
    z_{\kappa}(m) \ :=\ -\frac{1}{m}+\frac{1}{\ration}\int_{\ratiop^{-\kappa}}^{\infty}\frac{1}{\kappa t^{1/\kappa}(1+t \cdot m)}\dif t.
\]
Based on \eqref{eq:mz}, we can consider $z_{\kappa}(m)$ to be the inverse of $m_{\kappa}(z)$ wherever $m_{\kappa}(z)$ exists.
Then, note that
\BE
    \frac{\dif z_{\kappa}(m)}{\dif m}
    &=&
    \frac{1}{m^2} - \frac{1}{\ration}\int_{\ratiop^{-\kappa}}^{\infty}\frac{t^2}{\kappa t^{1+1/\kappa}(1+t \cdot m)^2}\dif t.\n
\EE
Hence, we have
\[
   \frac{\dif z_{\kappa}(m)}{\dif m}\ \gtreqless\ 0\quad \Leftrightarrow \quad 1 \ \gtreqless \  \frac{1}{\ration}\int_{\ratiop^{-\kappa}}^{\infty}\frac{t^2}{\kappa t^{1+1/\kappa}(m^{-1}+t)^2}\dif t. 
\]
Note that $\frac{1}{\ration}\int_{\ratiop^{-\kappa}}^{\infty}\frac{t^2}{\kappa t^{1+1/\kappa}(m^{-1}+t)^2}\dif t$ is a increasing function of $m$ with \BE
    \frac{1}{\ration}\int_{\ratiop^{-\kappa}}^{\infty}\frac{t^2}{\kappa t^{1+1/\kappa}(m^{-1}+t)^2}\dif t&\rightarrow& 0 \quad \text{as}\ m\rightarrow 0 ; \n \\
   \frac{1}{\ration}\int_{\ratiop^{-\kappa}}^{\infty}\frac{t^2}{\kappa t^{1+1/\kappa}(m^{-1}+t)^2}\dif t&\rightarrow& \frac{1}{\ration} > 1 \quad \text{as}\ m\rightarrow \infty.\n 
\EE
Hence, there exists a constant $m_c$ such that for all $0<m<m_c$, the function $z_{\kappa}(m)$ is increasing on the interval $\intoo{0,m_c}$ and decreasing on $\intoo{m_c,\infty}$.
Furthermore, note that
\BE
    m\cdot z_{\kappa}(m)&=&\frac{1}{\ration}\int_{\ratiop^{-\kappa}}^{\infty}\frac{1}{\kappa t^{1/\kappa}(m^{-1}+t)}\dif t-1.
\EE
Evaluating this integral as $m \to 0^+$ and as $m \to +\infty$ shows that
\BE
  m\cdot z_{\kappa}(m)
  & \to &
  \begin{cases}
    -1 & \text{as $m \to 0^+$} , \\
    \frac1\ration - 1 > 0 & \text{as $m \to +\infty$} ,
  \end{cases}
\EE
which in turn implies
\BE
  z_{\kappa}(m)
  & \to &
  \begin{cases}
    -\infty & \text{as $m \to 0^+$} , \\
    0 & \text{as $m \to +\infty$} .
  \end{cases}
\EE
Therefore, $z_{\kappa}(m)$ is strictly increasing on $z\leq 0$. Further, for $z\in [0, z_{\kappa}(m_c)]$, there are two only solutions of $m$ satisfying \eqref{eq:mz}. Therefore, since $m_{\kappa}(z)$ is defined to be the smallest positive solution of \eqref{eq:mz}, the mapping between $z\in (-\infty, z_{\kappa}(m_c)]$ and $m\in (0, m_c]$ defined by $z_{\kappa}(m)$ and $m_{\kappa}(z)$ is continuous, one-to-one, and $z_{\kappa}(m_c) > 0$.
This shows that $m_{\kappa}(z)$ is well-defined and continuous at $z=0$.
Then, by continuity of the defining expression, we conclude that $m_{\kappa}'(0)$ exists.

Next, we use the chain rule to calculate the value of $m_{\kappa}'(0)$.
From the definition of $m_{\kappa}$ in \eqref{eq:mz}, the change-of-variable in \eqref{eq:transformeq}, and the definition of $q_{\kappa}$ in \eqref{eq:q}, we have
\BE
	-z&=&\frac{1}{\ration\ratiop\cdot m_{\kappa}(z)^{1-1/\kappa}} \cdot q_{\kappa}\left(\frac{m_{\kappa}(z)^{1/\kappa}}{\ratiop},\ratiop\right) \label{eq:monotoeq1}
\EE
for $z$ in a neighborhood of $z=0$.
Also, from the analysis in Section \ref{sec:proofp>n}, we have $m_{\kappa}(0)=(s^*_{\kappa}\ratiop)^{\kappa}$ and $q_{\kappa}(s^*_{\kappa},\ratiop)=0$. Then, taking the derivative with respect to $z$ on both sides of \eqref{eq:monotoeq1} and with the chain rule, we have
\BE
	-1&=&\left(\frac{1}{\kappa}-1\right) \cdot \frac{m_{\kappa}(z)^{1/\kappa-2}}{\ration\ratiop} \cdot q_{\kappa}\left(\frac{m_{\kappa}(z)^{1/\kappa}}{\ratiop},\ratiop\right) \n \\
    && +\frac{1}{\ration\ratiop m_{\kappa}(z)^{1-1/\kappa}} \cdot \frac{\partial q_{\kappa}\left(s,\ratiop\right)}{\partial s}\Big|_{s=\frac{m_{\kappa}(z)^{1/\kappa}}{\ratiop}}\cdot\frac{m_{\kappa}(z)^{1/\kappa-1}}{\kappa\ratiop}\cdot m_{\kappa}'(z).\n
\EE
Hence, plugging in $z=0$ and solving for $m_{\kappa}'(0)$ gives
\[
	m_{\kappa}'(0)\=\frac{\kappa\ration\ratiop^2 (m_{\kappa}(0))^{2-2/\kappa}}{-\frac{\partial q_{\kappa}(s,\ratiop)}{\partial s}\big|_{s=s^*_{\kappa}}}
\]
Then, using the formula for the derivative of $q_{\kappa}$ in \eqref{eq:qpartial}, we have
\BE
	m_{\kappa}'(0)\=\kappa\ration m^2_{\kappa}(0)\cdot\frac{1+(s^*_{\kappa})^{\kappa}}{\ration-(\ratiop-\ration)(s^*_{\kappa})^{\kappa}}
	.\label{eq:m**}
\EE
Since $(s^*_\kappa)^\kappa < \ration/(\ratiop-\ration)$ (recall the argument in \Cref{sec:proofp>n} following Equation~\eqref{eq:qpartial}), it follows that $m_{\kappa}'(0) > 0$.

\subsection{Analysis of part 1} \label{app:part1}

In this section, we will prove that
\BE
	\trace{\vSigmaP\left(\vI-\vPi_{\vXP}\right)}
	&\Ipt&
	\frac{N^{1-\kappa}\ration}{m_{\kappa}(0)}
	.	\label{eq:p>npart1}
\EE
(The existence and uniqueness of $m^*_{\kappa} := m_{\kappa}(0)$ is proved in the beginning of \Cref{sec:proofp>n}.)
Let $\vtSigmaP=N^{\kappa}\vSigmaP$ and $\vtXP=N^{\kappa/2}\vXP$, then we have, for all $\mu > 0$,
\BE
	\trace{\vSigmaP\left(\vI-\vPi_{\vXP}\right)}
	&=&
	\frac{n}{N^{\kappa}}\left(\frac{1}{n}\trace{\vtSigmaP}-\frac{1}{n}\trace{\vtSigmaP \vtXP^{\t}\left(\vtXP\vtXP^{\t}\right)^{-1}\vtXP}\right)
		\n\\
	&=&
	\frac{n}{N^{\kappa}}\left(\frac{1}{n}\trace{\vtSigmaP}-\frac{1}{n}\trace{\vtSigmaP \left(\vtXP^{\t}\vtXP+\mu n\vI\right)^{-1}\vtXP^{\t}\vtXP}+\epsilon_{\mu_n}\right)
		\n\\
	&=&
	\frac{n}{N^{\kappa}}\left(\mu\cdot\frac{1}{n}\trace{\vtSigmaP \left(\frac{1}{n}\vtXP^{\t}\vtXP+\mu \vI\right)^{-1}}+\epsilon_{\mu_n}\right)
	, \label{eq:p>neq_2}
\EE
where $\epsilon_{\mu_n}$ is given by
\[
	\epsilon_{\mu_n}
	\ := \
	\frac{1}{n}\trace{\vtSigmaP\left(\vtXP^{\t}\vtXP+n\mu\vI\right)^{-1}\vtXP^{\t}\vtXP}-\frac{1}{n}\trace{\vtSigmaP\vtXP^{\t}\left(\vtXP\vtXP^{\t}\right)^{-1}\vtXP}
	.	
\] 
Since $n/N^{\kappa} \to N^{1-\kappa}\ration$, the claim in \eqref{eq:p>npart1} is implied by
\[
	\mu\cdot\frac{1}{n}\trace{\vtSigmaP \left(\frac{1}{n}\vtXP^{\t}\vtXP+\mu \vI\right)^{-1}}+\epsilon_{\mu_n}
	\=
	\frac{1}{m_{\kappa}(0)}+\op(1)
	.
\]
Hence, our task is reduced to finding a suitable positive sequence $(\mu_n)_{n\geq1}$ such that the following hold:
\BE
	|\epsilon_{\mu_n}|
	&=& 
	\op(1)
	,\label{eq:p>ngoal1}
\EE
and
\BE
	\mu_n\cdot\frac{1}{n}\trace{\vtSigmaP \left(\frac{1}{n}\vtXP^{\t}\vtXP+\mu_n \vI\right)^{-1}} &\ipt& \frac{1}{m_{\kappa}(0)}
	.	\label{eq:p>ngoal2}
\EE
With foresight, we shall assume that
\[ \mu_n<\min\left\{\frac{1}{\sqrt{N}},o(N^{-\kappa})\right\} . \]

\subsubsection{Proof of Equation~\eqref{eq:p>ngoal1}}

Let us first show \eqref{eq:p>ngoal1}.
Towards this end, we bound $|\epsilon_{\mu_n}|$ as follows:
\BE
	|\epsilon_{\mu_n}|
	&=&
	\frac{1}{n}\left|\trace{\vtSigmaP\left(\left(\vtXP^{\t}\vtXP+\mu_n n\vI\right)^{-1}\vtXP^{\t}\vtXP-\vtXP^{\t}\left(\vtXP\vtXP^{\t}\right)^{-1}\vtXP\right)}\right|
		\n\\
	&\stackrel{\text{(i)}}{\leq}&
	\frac{1}{n}\|\vtSigmaP\|_2\trace{\vtXP^{\t}\left(\vtXP\vtXP^{\t}\right)^{-1}\vtXP-\left(\vtXP^{\t}\vtXP+\mu_n n\vI\right)^{-1}\vtXP^{\t}\vtXP}
		\n\\
	&\leq&
	\frac{N^{\kappa}}{n}\cdot \sum_{i=1}^n\frac{\mu_n}{\tsigma_i+\mu_n}
	\=
	N^{\kappa}\cdot \mu_n\cdot m_n(-\mu_n)
	\ \leq \
	N^{\kappa}\cdot \frac{\mu_n}{\min_i(\tsigma_i)}
	,	\label{eq:elambda}
\EE
where $\tsigma_i$ is the $i$-th eigenvalue of $\frac{1}{n}\vtXP\vtXP^{\t}$ and $m_n(z)$ is the Stieltjes transform of the empirical eigenvalue distribution of $\frac{1}{n}\vtX\vtX^{\t}$. Inequality (i) holds because
\[
\vtXP^{\t}\left(\vtXP\vtXP^{\t}\right)^{-1}\vtXP-\left(\vtXP^{\t}\vtXP+\mu_n n\vI\right)^{-1}\vtXP^{\t}\vtXP\] 
is positive semi-definite.
Hence, the proof of \eqref{eq:p>ngoal1} only require us to lower bound $\min_{i}(\tsigma_i)$ and the following lemma will help us complete this task.

\begin{lemma}\label{lem:St}
Suppose the empirical eigenvalue distribution of the diagonal matrix $\vH$ converges to a limiting distribution $\mathcal{H}$ with probability density function $f_h$.
Assume that the support of $f_h$ is a subset of the interval $[\eta_1,\infty)$ for some positive constant $\eta_1$.
Let $\vbX\in \bbR^{n\times p}$ be a standard Gaussian matrix and suppose $p/n\rightarrow \gamma>1$.
Let $m_n(z)$ be the Stieltjes transform of the empirical eigenvalue distribution $\mathcal{F}_n$ of $\frac{1}{n}\vbX\vH\vbX^{\t}$.
Then $\mathcal{F}_n$ converges to a limit $\mathcal{F}$ whose Stieltjes transform, denoted by $m(z)$, satisfies
\BE
	m(z)&=& -\left(z-\gamma \int_{\eta_1}^{\infty} \frac{t f_h(t)\dif t}{1+t\cdot m(z)}\right)^{-1} , \quad \forall z\in \operatorname{supp}(\mathcal{F})^c.	\label{eq:lem_m}
\EE
Further, there exists a constant $c_{\epsilon}>0$ such that the minimum eigenvalue of $\frac{1}{n}\vbX\vH\vbX^{\t}$ is lower-bounded by $c_{\epsilon}$ in probability. Finally, for any increasing sequence $z_n\rightarrow 0^{-}$, we have 
\BE
	m_n(z_n) \ \ipt \ m(0)  \quad \text{and} \quad m_n'(z_n) \ \ipt \ m'(0). 
\EE	
\end{lemma}

The proof of \Cref{lem:St} is shown in \Cref{app:St}. Hence to apply Lemma \ref{lem:St}, we need the empirical distribution of the eigenvalues of the covariance matrix $\vSigmaP$ converges to a limiting distribution and thus we need to scale $\vSigmaP$ properly. The following \namecref{lem:dH} confirms that the correct scaling is $p^{\kappa}$.

\begin{lemma}\label{lem:dH}
  Let $S=\{i\}_{p_1< i \leq p_2}$ with $0\leq p_1<p_2\leq N$. Suppose $\frac{p_1}{N}\rightarrow \alpha_1$ and $\frac{p_2}{N}\rightarrow \alpha_2$ with $0\leq \alpha_1<\alpha_2\leq 1$. Then, the empirical eigenvalue distribution of $N^{\kappa}\vSigma_{S}$ converges to a (non-random) distribution $\mathcal{F}$ with probability density function $f$ given by 
\BE
	f(s) \= \left\{\begin{aligned}
	&\frac{1}{\kappa(\alpha_2-\alpha_1)}s^{-1-\frac{1}{\kappa}}\cdot \ind{s \in \left[\alpha_2^{-\kappa},\alpha_1^{-\kappa}\right]}, && \alpha_1>0\\
	&\frac{1}{\kappa\alpha_2}s^{-1-\frac{1}{\kappa}}\cdot \ind{s \in [\alpha_2^{-\kappa},\infty)}, && \alpha_1=0
	\end{aligned}\right.
	.	\label{eq:pdf}
\EE
\end{lemma}

The proof of \Cref{lem:dH} is shown in \Cref{app:dH}.
Using \Cref{lem:St}, \Cref{lem:dH}, and \eqref{eq:elambda},
we see that since $\mu_n = o(N^{-\kappa})$, we have
\[ |\epsilon_{\mu_n}| = \op(1) , \]
which establishes Equation~\eqref{eq:p>ngoal1}.

\subsubsection{Proof of Equation~\eqref{eq:p>ngoal2}}
\label{hardpart}

Our next goal is to prove \eqref{eq:p>ngoal2}, i.e.,
\[
	\frac{\mu_n}{n}\tr{\vtSigmaP \vtS_n} \ipt \frac{1}{m_{\kappa}(0)}
\]
where
\[ \vtS_n := \left(\frac1n\vtXP^{\t}\vtXP+\mu_n\vI\right)^{-1} . \]
The same result has been proved in Lemma 2.2 of \cite{ledoit2011eigenvectors} with additional assumption that
the empirical eigenvalue distribution of $\vtSigma$
converges to a limiting distribution with bounded support.
However, this assumption does not hold in our case.
We employ a similar proof strategy with more involved arguments based on leave-one-out estimates~\citep{xu2019consistent}.

Let $\vtx_i$ be the $i$-th row of $\vtXP$.
Then using the identity
\[ \vtS_n^{-1}-\mu_n\vI=\frac{1}{n}\sum_{i=1}^n\vtx_i\vtx_i^{\t} , \]
we have
\BE
	\frac{1}{n}\sum_{i=1}^n\vtx_i^{\t}\vtS_n\vtx_i
	&=&
	\frac{1}{n}\trace{\sum_{i=1}^n\vtS_n\vtx_i\vtx_i^{\t}}
		\n\\
	&=&
	\trace{\vtS_n(\vtS_n^{-1}-\mu_n\vI)}
		\n\\	
	&=&
	\trace{\vI-\mu_n \vtS_n}
	.	\label{eq:lemSt_eq1}
\EE
For each $i=1,\dotsc,n$, define
\[ \vtS_n^{\backslash i}:= \left( \frac1n\vtXP^{\t}\vtXP-\frac{1}{n}\vtx_i\vtx_i^{\t}+\mu_n\vI \right)^{-1} = \left( \vtS_n^{-1}-n^{-1}\vtx_i\vtx_i^{\t} \right)^{-1} . \]
By the Sherman-Morrison formula, we have
\BE
	\vtS_n
	&=&
	\vtS_n^{\backslash i}
	-\frac{1}{n}\cdot\frac{\vtS_n^{\backslash i}\vtx_i\vtx_i^{\t}\vtS_n^{\backslash i}}{1+\frac{1}{n}\vtx_i^{\t}\vtS_n^{\backslash i}\vtx_i}
	.	\label{eq:p>npart1bi}
\EE
Hence, with \eqref{eq:lemSt_eq1}, we have
\BE
	\tr{\vI-\mu_n \vtS_n}
	&=&
	\frac{1}{n}\sum_{i=1}^n\vtx_i^{\t}\vtS_n\vtx_i
	\=
	\frac{1}{n}\sum_{i=1}^n\vtx_i^{\t}\left(\vtS_n^{\backslash i}
	-\frac{1}{n}\cdot\frac{\vtS_n^{\backslash i}\vtx_i\vtx_i^{\t}\vtS_n^{\backslash i}}{1+\frac{1}{n}\vtx_i^{\t}\vtS_n^{\backslash i}\vx_i}\right)\vtx_i
		\n\\
	&=&
	\frac{1}{n}\sum_{i=1}^n\frac{\vtx_i^{\t}\vtS_n^{\backslash i}\vtx_i}{1+\frac{1}{n}\vtx_i^{\t}\vtS_n^{\backslash i}\vtx_i}
	\=
	n-\sum_{i=1}^n\frac{1}{1+\frac{1}{n}\vtx_i^{\t}\vtS_n^{\backslash i}\vtx_i}
	.	\n
\EE
Since $\tr{\vI-\mu_n \vtS_n}=n-n\mu_n \cdot m_n(-\mu_n)$, we have
\BE
	m_n(-\mu_n)&=&\frac{1}{n}\sum_{i=1}^n\frac{1}{\mu_n+\frac{\mu_n}{n}\vtx_i^{\t}\vtS_n^{\backslash i}\vtx_i}
	.	\label{eq:p>neq_equal}
\EE
Note that $|m_n(-\mu_n)-m_n(0)|\leq \frac{\mu_n}{\min(\tsigma_i^2)}$ where $\tsigma_i$ is the $i$th eigenvalue of $\frac{1}{n}\vtXP\vtXP^{\t}$. By \Cref{lem:St}, we have 
\BE
	m_n(-\mu_n)\= m_n(0)+\Op(\mu_n)\ \ipt \ m_{\kappa}(0). \label{eq:msmallterm}
\EE
Therefore, the LHS of \eqref{eq:p>neq_equal} converges to $m_{\kappa}(0)$ in probability. Then we just need to show the RHS of \eqref{eq:p>neq_equal} converges to \[ \left(\frac{\mu_n}{n}\trace{\vtSigmaP\vtS_n}\right)^{-1} \] in probability. Let 
\[
	\Delta_i
	\ := \ \frac{\mu_n}{n}\trace{\vtSigmaP\vtS_n}-\frac{\mu_n}{n}\vtx_i^{\t}\vtS_n^{\backslash i}\vtx_i-\mu_n,
\]
then note that
\BE
  \left|\left(\frac{\mu_n}{n}\trace{\vtSigmaP\vtS_n}\right)^{-1}-m_n(-\mu_n)\right|
  & = &
	\left|\left(\frac{\mu_n}{n}\trace{\vtSigmaP\vtS_n}\right)^{-1}-\frac{1}{n}\sum_{i=1}^n\frac{1}{\mu_n+\frac{\mu_n}{n}\vtx_i^{\t}\vtS_n^{\backslash i}\vtx_i}\right|
		\n\\
	&=&
	\left|\frac{1}{n}\sum_{i=1}^n\frac{\Delta_i}{\frac{\mu_n}{n}\trace{\vtSigmaP\vtS_n}\cdot\left(\frac{\mu_n}{n}\trace{\vtSigmaP\vtS_n}-\Delta_i\right)}\right|
		\n\\
	&\leq&
	\sup_i \frac{|\Delta_i|}{\frac{\mu_n}{n}\trace{\vtSigmaP\vtS_n}\cdot\left|\frac{\mu_n}{n}\trace{\vtSigmaP\vtS_n}-|\Delta_i|\right|}
	.	\n
\EE
We claim that
\BE
  \frac{\mu_n}{n}\trace{\vtSigmaP\vtS_n} &=& \Thp(1) ; \n \\
  \sup_i|\Delta_i| &=& \Op\left(\frac{\ln N}{\sqrt{N}}\right) \n
\EE
(\Cref{claim:p>ngoal4} and \Cref{claim:p>ngoal5} below).
Then with \eqref{eq:msmallterm}, we have
\[
	\left(\frac{\mu_n}{n}\trace{\vtSigmaP\vtS_n}\right)^{-1}\ \ipt \ m_{\kappa}(0).
\]
This in turn implies Equation~\eqref{eq:p>ngoal2} as desired.

\subsubsection{Supporting propositions}

\begin{proposition}
  \label{claim:p>ngoal4}
  \[ \frac{\mu_n}{n}\trace{\vtSigmaP\vtS_n} \= \Thp(1) . \]
\end{proposition}
\begin{proof}
  Note that 
\BE
	\frac{\mu_n}{n}\trace{\vtSigmaP\vtS_n}
	&\stackrel{\mathrm{(i)}}{\geq}&
	\frac{\mu_n}{n}\trace{\vtS_n}
	\=
	\frac{\mu_n}{n}\trace{\left(\frac{1}{n}\vtXP^{\t}\vtXP+\mu_n\vI\right)^{-1}}
		\n\\
	&\stackrel{\mathrm{(ii)}}{\geq}&
	\frac{\mu_n}{n}\cdot\frac{p-n}{\mu_n}
	\ \rightarrow \ \frac{\ratiop-\ration}{\ration}\ >\ 0
	,	\n
\EE
where inequality (i) holds due to the fact that $\vtSigmaP$ is a diagonal matrix with diagonal elements lower bounded by $1$, and inequality (ii) holds due to the fact that $\left(\frac{1}{n}\vtXP^{\t}\vtXP+\mu_n\vI\right)^{-1}$ has at least $p-n$ number of eigenvalues $\frac{1}{\mu_n}$. Hence, we have $\frac{\mu_n}{n}\tr{\vtSigmaP\vtS_n}=\Omp(1)$. To show $\frac{\mu_n}{n}\trace{\vtSigmaP\vtS_n}=\Op(1)$ as well, let us introduce $\vbS_n=\vtSigmaP^{1/2}\vS_n\vtSigmaP^{1/2}$, then we have
\[
	\frac{\mu_n}{n}\trace{\vtSigmaP\vtS_n} \= \frac{\mu_n}{n}\trace{\vbS_n} \ \leq \mu_n \frac{p}{n}\|\vbS_n\|_2.
\]
Therefore, as $p/n\rightarrow \ratiop/\ration$, we just need to upper bound $\|\vbS_n\|_2$.
To do this, we use the following \namecref{lem:eigen}.
 
\begin{lemma}\label{lem:eigen}
Let $\vSigma\in \bbR^{p\times p}$ be a diagonal matrix. Let $\vbX\in \bbR^{n\times p}$ be a standard Gaussian matrix with $p>n$. Suppose $\frac{p}{n}\rightarrow \gamma>1$ as $n,p\rightarrow \infty$. Suppose the $\frac{n}{2}$th smallest diagonal element of $\vSigma$ can be lower bounded by a constant $\nu$ with probability $1-\delta$. Then the minimum eigenvalue of $\frac{1}{n}\vbX^{\t}\vbX+\mu \vSigma$ is lower bounded by 
\[
	\min\left(c_1, c_2 \mu\right)
\]
with probability $1-cn^2\cdot \mathrm{exp}(-c'n)-\delta$ for some positive constants $c_1,c_2, c,c'>0$ that only depend on $\gamma$.
\end{lemma}

The proof of Lemma \ref{lem:eigen} is shown in Appendix \ref{app:eigen}. Note that 
\[
	\vbS_n=\left(\frac{1}{n}\vbXP^{\t}\vbXP+\mu_n\vtSigmaP^{-1}\right)^{-1}
\]
where $\vbXP=\vtXP\vtSigmaP^{-1/2}$ is a standard Gaussian matrix. Further, the $\frac{n}{2}$ smallest eigenvalue of $\vtSigmaP^{-1}$ is $\frac{n^{\kappa}}{(2 p)^{\kappa}}$ which converges to a constant $(\frac{\ration}{2\ratiop})^{\kappa}$. Hence, by \Cref{lem:eigen}, we know $\|\vbS_n\|_2$ is upper bounded by $\Op(\frac{1}{\mu_n})$ and thus, $\frac{\mu_n}{n}\trace{\vtSigmaP\vtS_n}=\Op(1)$.
This completes the proof of \Cref{claim:p>ngoal4}.
\end{proof}

\begin{proposition}
  \label{claim:p>ngoal5}
  \[ \sup_i|\Delta_i| \= \Op\left(\frac{\ln N}{\sqrt{N}}\right) . \]
\end{proposition}
\begin{proof}
Let us introduce $\vbS_n^{\backslash i}=\vtSigmaP^{1/2}\vS_n^{\backslash i}\vtSigmaP^{1/2}$ and $\vbx_i=\vtSigmaP^{-1/2}\vtx_i$. Then,  
\BE
	\vbS_n^{\backslash i}=\left(\frac{1}{n}\vbXP^{\t}\vbXP-\frac{1}{n}\vbx_i\vbx_i^{\t}+\mu_n\vtSigmaP^{-1}\right)^{-1}
	,	\label{eq:36}
\EE
where $\vbx_i$ is the $i$th row of $\vbXP$. Further, we have 
\[
	\Delta_i = \frac{\mu_n}{n}\trace{\vbS_n}-\frac{\mu_n}{n}\vbx_i^{\t}\vbS_n^{\backslash i}\vbx_i-\mu_n.
\] 
To bound $|\Delta_i|$, we can decompose $\Delta_i$ into three parts:
\[
	\Delta_i \= \left(\frac{\mu_n}{n}\trace{\vbS_n}-\frac{\mu_n}{n}\trace{\vbS_n^{\backslash i}}\right)+\left(\frac{\mu_n}{n}\trace{\vbS_n^{\backslash i}}-\frac{\mu_n}{n}\vbx_i^{\t}\vbS_n^{\backslash i}\vbx_i\right)-\mu_n
\]
Intuitively, the first part should be small since $\vbS_n$ and $\vbS_n^{\backslash i}$ only differ at one sample. For the second part, since $\vbx_i$ is independent of $\vbS_n^{\backslash i}$, the law of large numbers implies that it should be small as well. Finally, we have $\mu_n\rightarrow 0$. We now make these arguments rigorous. By \Cref{lem:eigen} again, we have 
\BE
	\max\left(\|\vbS_n\|_2, \max_{i}\|\vbS_n^{\backslash i}\|_2\right)
	&\leq& \Op\left(\frac{1}{\mu_n}\right)
	.	\label{eq:maximum}
\EE
Then, we can show that the difference between $\frac{\mu_n}{n}\trace{\vbS_n}$ and $\frac{\mu_n}{n}\trace{\vbS_n^{\backslash i}}$ is small. Note that, by the Sherman-Morrison formula,
\BE
	\sup_i\left|\frac{\mu_n}{n}\trace{\vbS_n}-\frac{\mu_n}{n}\trace{\vbS_n^{\backslash i}}\right|
	&=&
	\sup_i\left|\frac{\mu_n}{n}\trace{\frac{\vbS_n^{\backslash i}\vbx_i\vbx_i^{\t}\vbS_n^{\backslash i}}{n+\vbx_i^{\t}\vbS_n^{\backslash i}\vbx_i}}\right|
	\=
	\sup_i\frac{1}{n}\frac{\mu_n\vbx_i^{\t}\left(\vbS_n^{\backslash i}\right)^2\vbx_i}{n+\vbx_i^{\t}\vbS_n^{\backslash i}\vbx_i}
		\n\\
	\ < 
	\sup_i\frac{1}{n} \frac{\mu_n\vbx_i^{\t}\left(\vbS_n^{\backslash i}\right)^2\vbx_i}{\vbx_i^{\t}\vbS_n^{\backslash i}\vbx_i}
	& \leq &
	\sup_i\frac{\mu_n}{n}\cdot \Op\left(\frac{1}{\mu_n}\right)\cdot \frac{\vbx_i^{\t}\vbS_n^{\backslash i}\vbx_i}{\vbx_i^{\t}\vbS_n^{\backslash i}\vbx_i}
	\=
	\Op\left(\frac{1}{n}\right)
	.	\n
\EE
Then we want to show the difference between $\frac{\mu_n}{n}\trace{\vbS_n^{\backslash i}}$ and $\frac{\mu_n}{n}\vbx_i^{\t}\vbS_n^{\backslash i}\vbx_i$ is small. Note that $\vbx_i^{\t}$ is a standard Gaussian vector and it is independent of $\vbS_n^{\backslash i}$. Hence, the expectation of $\frac{\mu_n}{n}\vbx_i^{\t}\vbS_n^{\backslash i}\vbx_i$ is given by $\frac{\mu_n}{n}\trace{\vbS_n^{\backslash i}}$. Further, by standard $\chi^2$ tail bounds \citep{laurent2000adaptive}, we have
\BE
	\P\left(\max_{i}\left|\frac{\mu_n}{n}\vbx_i^{\t}\vbS_n^{\backslash i}\vbx_i-\frac{\mu_n}{n}\trace{\vbS_n^{\backslash i}}\right|\geq \frac{2\mu_n p}{n}(\epsilon+\epsilon^2)\|\vbS_n^{\backslash i}\|\right)
	&\leq& e^{-\epsilon^2 p}
	.	\label{eq:lemSt_eq3}
\EE
Choose $\epsilon=\frac{\log n}{\sqrt{p}}$, we know 
\BE
	\sup_i\left|\frac{\mu_n}{n}\vbx_i^{\t}\vbS_n^{\backslash i}\vbx_i-\frac{\mu_n}{n}\trace{\vbS_n^{\backslash i}}\right|
	\=
	\Op\left(\frac{\ln N}{\sqrt{N}}\right)
	. \label{eq:similarpart1}
\EE
Hence, we have
\[
	|\Delta_i|\leq \sup_i\left|\frac{\mu_n}{n}\trace{\vbS_n}-\frac{\mu_n}{n}\trace{\vbS_n^{\backslash i}}\right|+\sup_i\left|\frac{\mu_n}{n}\vbx_i^{\t}\vbS_n^{\backslash i}\vbx_i-\frac{\mu_n}{n}\trace{\vbS_n^{\backslash i}}\right|+|\mu_n|
	\= \Op\left(\frac{\ln N}{\sqrt{N}}\right).
\]
\end{proof}

\subsection{Analysis of part 2}\label{app:part2}

In this section, we will prove that
\BE
  \text{part 2} 
	\ \ipt \ 
	N^{1-\kappa}\cdot \frac{m_{\kappa}'(0)}{m^2_{\kappa}(0)} \cdot \int_{\ratiop}^1t^{\kappa-2}\dif t+\op(N^{1-\kappa})
  .
  \n
\EE
We apply a proof similar to that of \Cref{thm:p<n} in \Cref{app:p<n}.
The conditional expectation of part 2 given $\vXP$ is 
\BE
  \bbE[\text{part 2} \mid \vXP ]&=&\trace{\vSigmapC}\cdot\left(\trace{\vSigmaP\vXP^{\t}\left(\vXP\vXP^{\t}\right)^{-2}\vXP}+1\right)
	.	\label{eq:p>npart2_expect}
\EE
(This expectation only conditions on $\vXP$; in particular, it averages over $\vXpC$.)
The variance of part 2 given $\vXP$ is
\BE
	\var\ (\text{part 2}~|\vXP) 
	&\leq&
	2\cdot\trace{\vSigmapC^2}\cdot\left\|\left(\vXP\vXP^{\t}\right)^{-1}\vXP\vSigmaP\vXP^{\t}\left(\vXP\vXP^{\t}\right)^{-1}\right\|_F^2
		\n\\
	&=&
	2\cdot\trace{\vSigmapC^2}\cdot\trace{\left(\vSigmaP\vXP^{\t}\left(\vXP\vXP^{\t}\right)^{-2}\vXP\right)^{2}}
	.	\label{eq:p>npart2var}
\EE
Let
\[ \psi := \trace{\vSigmaP\vXP^{\t}\left(\vXP\vXP^{\t}\right)^{-2}\vXP} . \]
Then by Markov's inequality, we have
\BE
	\text{part 2}&=& \trace{\vSigmapC}\cdot (\psi+1)+\Op\left(\frac{\sqrt{N-p}}{N^{\kappa}}\cdot \psi\right)
	.	\label{eq:p>npart2}
\EE
By \eqref{eq:tracep<n2}, we have
\BE
	\trace{\vSigmapC} 
	&\rightarrow&
	N^{1-\kappa}\int_{\ratiop}^1t^{-\kappa}\dif t
	.\n
\EE
Hence, we just need to show
\BE
	\psi+1&\ipt&\frac{m_{\kappa}'(0)}{m^2_{\kappa}(0)}, \label{eq:p>npart2goal}
\EE
as this will imply
\[
	\text{part 2} \ \ipt \ N^{1-\kappa}\cdot \frac{m_{\kappa}'(0)}{m^2_{\kappa}(0)} \cdot \int_{\ratiop}^1t^{-\kappa}\dif t+\op(N^{1-\kappa})
\]
as required.

To prove \eqref{eq:p>npart2goal}, let us first rescale $\vSigma$ to $\vtSigma$ and introduce the positive sequence $(\mu_n)_{n\geq1}$ just like what we did for part 1, and with foresight, we pick the sequence such that
\[ \mu_n = o(N^{-\kappa}) . \]
Then we have
\BE
	\psi+1
	&=& 
	\trace{\vtSigmaP\vtXP^{\t}\left(\vtXP\vtXP^{\t}\right)^{-2}\vtXP}+1
		\n\\
	&=&
	\frac{1}{n}\trace{\vtSigmaP\left(\frac{1}{n}\vtXP^{\t}\vtXP+\mu_n\vI\right)^{-1}\left(\frac{1}{n}\vtXP^{\t}\vtXP\right)\left(\frac{1}{n}\vtXP^{\t}\vtXP+\mu_n\vI\right)^{-1}}+\epsilon'_{\mu_n}+1
		\n\\
	&=&
	\frac{1}{n}\trace{\vtSigmaP\vtS_n}-\frac{\mu_n}{n}\trace{\vtSigmaP\vtS_n^2}+1+\epsilon'_{\mu_n}
	,	\n
\EE
where $\epsilon'_{\mu_n}$ is given by
\BE
  \epsilon'_{\mu_n}
	&=&
	\frac{1}{n}\trace{\vtSigmaP\frac{1}{\sqrt{n}}\vtXP^{\t}\left(\frac{1}{n}\vtXP\vtXP^{\t}\right)^{-2}\frac{1}{\sqrt{n}}\vtXP}
		\n\\
	&&-\frac{1}{n}\trace{\vtSigmaP\left(\frac{1}{n}\vtXP^{\t}\vtXP+\mu_n\vI\right)^{-1}\left(\frac{1}{n}\vtXP^{\t}\vtXP\right)\left(\frac{1}{n}\vtXP^{\t}\vtXP+\mu_n\vI\right)^{-1}}
	.	\n	
\EE
We shall prove the following:
\BE
  |\epsilon'_{\mu_n}|
  & = & \op(1) ,
  \label{eq:p>npart2goal0}
\EE
and
\BE
	\frac{1}{n}\trace{\vtSigmaP\vtS_n}-\frac{\mu_n}{n}\trace{\vtSigmaP\vtS_n^2}+1
	&\ipt&
	\frac{m_{\kappa}'(0)}{m^2_{\kappa}(0)}
	,	\label{eq:p>npart2goal1}
\EE
which suffices to establish \eqref{eq:p>npart2goal}.

\subsubsection{Proof of Equation~\eqref{eq:p>npart2goal0}}

To bound $|\epsilon_{\mu_n}'|$, note that
\BE
	|\epsilon_{\mu_n}'|
	&\stackrel{\text{(i)}}{\leq}&
	\frac{1}{n}\|\vtSigmaP\|_2\left(\trace{\frac{1}{\sqrt{n}}\vtXP^{\t}\left(\frac{1}{n}\vtXP\vtXP^{\t}\right)^{-2}\frac{1}{\sqrt{n}}\vtXP}\right.
		\n\\
	&&
	\left.-\trace{\left(\frac{1}{n}\vtXP^{\t}\vtXP+\mu_n\vI\right)^{-1}\left(\frac{1}{n}\vtXP^{\t}\vtXP\right)\left(\frac{1}{n}\vtXP^{\t}\vtXP+\mu_n\vI\right)^{-1}}\right)
		\n\\
	&\leq&
	\frac{N^{\kappa}}{n}\cdot \sum_{i=1}^n\frac{\mu_n(2\tsigma_i+\mu_n)}{(\tsigma_i+\mu_n)^2\tsigma_i}
	\ \leq \
	2N^{\kappa}\cdot \frac{\mu_n}{\min_i(\tsigma_i^2)}
	,	\label{eq:elambdapart2}
\EE
where $\tsigma_i$ is the $i$-th eigenvalue of $\frac{1}{n}\vtXP\vtXP^{\t}$ and inequality (i) holds due to the fact that
\[
\frac{1}{\sqrt{n}}\vtXP^{\t}\left(\frac{1}{n}\vtXP\vtXP^{\t}\right)^{-2}\frac{1}{\sqrt{n}}\vtXP-\left(\frac{1}{n}\vtXP^{\t}\vtXP+\mu_n\vI\right)^{-1}\left(\frac{1}{n}\vtXP^{\t}\vtXP\right)\left(\frac{1}{n}\vtXP^{\t}\vtXP+\mu_n\vI\right)^{-1} \] 
is positive semi-definite.
By \Cref{lem:St}, \Cref{lem:dH}, and \eqref{eq:elambdapart2}, since $\mu_n=o(N^{-\kappa})$, we have
\BE
	|\epsilon'_{\mu_n}| &=& \op(1).	\n 
\EE

\subsubsection{Proof of Equation~\eqref{eq:p>npart2goal1}}

We now prove
\BE
	\frac{1}{n}\trace{\vtSigmaP\vtS_n}-\frac{\mu_n}{n}\trace{\vtSigmaP\vtS_n^2}+1
	&\ipt&
	\frac{m_{\kappa}'(0)}{m^2_{\kappa}(0)}
	. \n
\EE
Towards this goal, we employ a strategy similar to the proof of \eqref{eq:p>ngoal2}. Using the identity $\vtS_n^{-1}-\mu_n\vI=\frac{1}{n}\sum_{i=1}^n\vtx_i\vtx_i^{\t}$, we have
\BE
	\frac{1}{n}\sum_{i=1}^n\vtx_i^{\t}\vtS_n^2\vtx_i
	&=&
	\frac{1}{n}\trace{\sum_{i=1}^n\vtS_n^2\vtx_i\vtx_i^{\t}}
		\n\\
	&=&
	\trace{\vtS_n^2(\vtS_n^{-1}-\mu_n\vI)}
		\n\\	
	&=&
	\trace{\vtS_n-\mu_n \vtS_n^2}
	.	\n
\EE  
With \eqref{eq:p>npart1bi}, we have
\BE
	\trace{\vtS_n-\mu_n \vtS_n^2}
	&=&
	\frac{1}{n}\sum_{i=1}^n\vtx_i^{\t}\vtS_n^2\vtx_i \n
  \\
  & = &
	\frac{1}{n}\sum_{i=1}^n\vtx_i^{\t}\left(\vtS_n^{\backslash i}
	-\frac{1}{n}\cdot\frac{\vtS_n^{\backslash i}\vtx_i\vtx_i^{\t}\vtS_n^{\backslash i}}{1+\frac{1}{n}\vtx_i^{\t}\vtS_n^{\backslash i}\vtx_i}
\right)^2\vtx_i
		\n\\
	&=&
  \sum_{i=1}^n\frac{\frac{1}{n}\vtx_i^{\t}\del[2]{\vtS_n^{\backslash i}}^2\vtx_i}{\left(1+\frac{1}{n}\vtx_i^{\t}\vtS_n^{\backslash i}\vtx_i\right)^2}
	.	\label{eq:p>npart2eq_1}
\EE
Note that $\frac{1}{n}\trace{\vtS_n-\mu_n \vtS_n^2}=m_n(-\mu_n)-\mu_nm_n'(-\mu_n)$. With \eqref{eq:p>neq_equal} and \eqref{eq:p>npart2eq_1}, we have
\BE
	-\mu_nm_n'(-\mu_n)
	&=&
	\frac{1}{n}\sum_{i=1}^n\frac{\frac{1}{n}\vtx_i^{\t}\left(\vtS_n^{\backslash i}\right)^2\vtx_i}{\left(1+\frac{1}{n}\vtx_i^{\t}\vtS_n^{\backslash i}\vtx_i\right)^2}-m_n(-\mu_n)
		\n\\
	&=&
	\frac{1}{n}\sum_{i=1}^n\frac{\frac{1}{n}\vtx_i^{\t}\left(\vtS_n^{\backslash i}\right)^2\vtx_i}{\left(1+\frac{1}{n}\vtx_i^{\t}\vtS_n^{\backslash i}\vtx_i\right)^2}-\frac{1}{n}\sum_{i=1}^n\frac{1}{\mu_n+\frac{\mu_n}{n}\vtx_i^{\t}\vtS_n^{\backslash i}\vtx_i}
		\n\\
	&=&
	\mu_n\cdot \frac{1}{n}\sum_{i=1}^n\frac{\frac{\mu_n}{n}\vtx_i^{\t}\left(\vtS_n^{\backslash i}\right)^2\vtx_i-1-\frac{1}{n}\vtx_i^{\t}\vtS_n^{\backslash i}\vtx_i}{\left(\mu_n+\frac{\mu_n}{n}\vtx_i^{\t}\vtS_n^{\backslash i}\vtx_i\right)^2}
	.	\n
\EE
Hence, we have
\BE
	m_n'(-\mu_n)
	&=&
	\frac{1}{n}\sum_{i=1}^n\frac{1+\frac{1}{n}\vtx_i^{\t}\left(\vtS_n^{\backslash i}-\mu_n\left(\vtS_n^{\backslash i}\right)^2\right)\vtx_i}{\left(\mu_n+\frac{\mu_n}{n}\vtx_i^{\t}\vtS_n^{\backslash i}\vtx_i\right)^2}
	.	\label{eq:p>npart2eq_2}
\EE
Note that
\[ |m_n'(-\mu_n)-m_n'(0)|\leq \frac{2\mu_n }{\min(\tilde{\lambda}_i^3)} , \]
where $\tsigma_1,\dotsc,\tsigma_n$ are the eigenvalues of $\frac{1}{n}\vtXP\vtXP^{\t}$.
Therefore, by \Cref{lem:St}, we have
\[
	m_n'(-\mu_n)\ = \  m_n'(0)+ \Op(\mu_n) \ \ipt \ m_{\kappa}'(0).
\]
From \eqref{eq:p>ngoal2}, \Cref{claim:p>ngoal4}, and \Cref{claim:p>ngoal5}, we know that
\BE
	\left(\mu_n+\frac{\mu_n}{n}\vtx_i^{\t}\vtS_n^{\backslash i}\vtx_i\right)^2
	&\ipt&
	\frac{1}{m^2_{\kappa}(0)}>0
	.	\label{eq:p>npart2eq_3}
\EE
We claim that
\BE
	\frac{1}{n}\trace{\vtSigmaP\vtS_n-\mu_n\vtSigmaP\vtS_n^2}
	&=&
	\Op(1) 
  ,
		\label{eq:p>npart2goal2}\\
	\frac{1}{n}\vtx_i^{\t}\left(\vtS_n^{\backslash i}-\mu_n\left(\vtS_n^{\backslash i}\right)^2\right)\vtx_i
	&=&
	\frac{1}{n}\trace{\vtSigmaP\vtS_n-\mu_n\vtSigmaP\vtS_n^2}+\Op\left(\frac{\ln N}{\sqrt{N}}\right)
	\label{eq:p>npart2goal3}
\EE
(\Cref{prop:p>npart2goal2} and \Cref{prop:p>npart2goal3} below).
So, we obtain from \eqref{eq:p>npart2eq_2}
\BE
m_n'(-\mu_n) \= \frac{1}{n}\sum_{i=1}^n\frac{1+\frac{1}{n}\vtx_i^{\t}\left(\vtS_n^{\backslash i}-\mu_n\left(\vtS_n^{\backslash i}\right)^2\right)\vtx_i}{\left(\mu_n+\frac{\mu_n}{n}\vtx_i^{\t}\vtS_n^{\backslash i}\vtx_i\right)^2}
	&\ipt&
	\frac{1+\frac{1}{n}\trace{\vtSigmaP\vtS_n-\mu_n\vtSigmaP\vtS_n^2}}{1/m_{\kappa}(0)^2}
	,	\n
\EE
i.e.,
\[
  \frac{m_n'(-\mu_n)}{m_\kappa(0)^2}
  \ipt
  1+\frac{1}{n}\trace{\vtSigmaP\vtS_n-\mu_n\vtSigmaP\vtS_n^2}
  .
\]
This suffices to prove \eqref{eq:p>npart2goal1} as required.

\subsubsection{Supporting propositions}

\begin{proposition}
  \label{prop:p>npart2goal2}
  \[ \frac{1}{n}\trace{\vtSigmaP\vtS_n-\mu_n\vtSigmaP\vtS_n^2}
	\=
	\Op(1) 
  .
\]
\end{proposition}
\begin{proof}
Recall that 
\[
	\vbS_n\=\vtSigmaP^{1/2}\vtS_n\vtSigmaP^{1/2}\=\left(\frac{1}{n}\vbXP^{\t}\vbXP+\mu_n\vtSigmaP^{-1}\right)^{-1}
	,
\]
where $\vbXP=\vtXP\vtSigmaP^{-1/2}$ is a standard Gaussian matrix. Let $\frac{1}{n}\vbXP^{\t}\vbXP=\vU\vLambda\vU^{\t}$ be the singular value decomposition of $\frac{1}{n}\vbXP^{\t}\vbXP$, where $\vU\vU^{\t}=\vI$ and $\vLambda$ is a diagonal matrix with 
\[ \Lambda_{1,1}\geq \Lambda_{2,2} \geq \dotsb \geq \Lambda_{n,n} \geq \Lambda_{n+1,n+1}=\dotsb=\Lambda_{p,p}=0 . \]
Hence, we have
\BE
	\vtSigmaP^{1/2}\vtS_n\vtSigmaP^{1/2}-\mu_n\vtSigmaP^{1/2}\vtS_n^2\vtSigmaP^{1/2}
	&=&
	\vbS_n\left(\frac{1}{n}\vbXP^{\t}\vbXP\right)\vbS_n
		\n\\
	&=&
	\left(\vLambda+\mu_n\vU^{\t}\vtSigmaP^{-1}\vU\right)^{-1}\vLambda\left(\vLambda+\mu_n\vU^{\t}\vtSigmaP^{-1}\vU\right)^{-1}
	.	\n
\EE
Our next step is to bound the maximum eigenvalue of 
\[	
	\left(\vLambda+\mu_n\vU^{\t}\vtSigmaP^{-1}\vU\right)^{-1}\vLambda\left(\vLambda+\mu_n\vU^{\t}\vtSigmaP^{-1}\vU\right)^{-1}
	.
\]
Let $\phi_n$ be the smallest eigenvalue of $\vtSigmaP^{-1}$. Define $\vLambda_{\phi}=\vLambda+\frac{\mu_n\phi_n}{2}\vI$ and $\vSigma_{\phi}^{-1}=\mu_n(\vtSigmaP^{-1}-\frac{\phi_n}{2}\vI)$. Then $\vLambda_{\phi}$ and $\vSigma_{\phi}$ are two positive definite diagonal matrices. Intuitively, for $\mu_n$ small enough,
\[
	\left(\vLambda+\mu_n\vU^{\t}\vtSigmaP^{-1}\vU\right)^{-1}\vLambda\left(\vLambda+\mu_n\vU^{\t}\vtSigmaP^{-1}\vU\right)^{-1}
  \approx \vLambda_{\phi}^{-1}\vLambda\vLambda_{\phi}^{-1} , \]
the latter having a maximum eigenvalue bounded by a constant.
We now make this argument rigorous. By the Sherman-Morrision formula, we have
\BE
	\left(\vLambda+\mu_n\vU^{\t}\vtSigmaP^{-1}\vU\right)^{-1}
  &=&
	\left(\vLambda_{\phi}+\vU^{\t}\vSigma_{\phi}^{-1}\vU\right)^{-1}
  \n
  \\
  &=&
	\vLambda_{\phi}^{-1}-\vLambda_{\phi}^{-1}\vU^{\t}\left(\vSigma_{\phi}+\vU^{\t}\vLambda_{\phi}^{-1}\vU\right)^{-1}\vU\vLambda_{\phi}^{-1}
	.	\n
\EE
Hence, we know
\BE
	\lefteqn{\left\|\left(\vLambda+\mu_n\vU^{\t}\vtSigmaP^{-1}\vU\right)^{-1}\vLambda\left(\vLambda+\mu_n\vU^{\t}\vtSigmaP^{-1}\vU\right)^{-1}\right\|_2}
		\n\\
	&\leq&
	2\left\|\vLambda_{\phi}^{-1}\vLambda\vLambda_{\phi}^{-1}\right\|_2
		\n\\
	&&+2\left\|\vLambda_{\phi}^{-1}\vU^{\t}\left(\vSigma_{\phi}+\vU^{\t}\vLambda_{\phi}^{-1}\vU\right)^{-1}\vU\vLambda_{\phi}^{-1}\vLambda\vLambda_{\phi}^{-1}\vU^{\t}\left(\vSigma_{\phi}+\vU^{\t}\vLambda_{\phi}^{-1}\vU\right)^{-1}\vU\vLambda_{\phi}^{-1}\right\|_2
		\n\\
	&\leq&
	2\left\|\vLambda_{\phi}^{-1}\vLambda\vLambda_{\phi}^{-1}\right\|_2\left(1+\left\|\vLambda_{\phi}^{-1}\vU^{\t}\left(\vSigma_{\phi}+\vU^{\t}\vLambda_{\phi}^{-1}\vU\right)^{-1}\vU\vU^{\t}\left(\vSigma_{\phi}+\vU^{\t}\vLambda_{\phi}^{-1}\vU\right)^{-1}\vU\vLambda_{\phi}^{-1}\right\|_2\right)
		\n\\
	&=&
	2\left\|\vLambda_{\phi}^{-1}\vLambda\vLambda_{\phi}^{-1}\right\|_2\left(1+\left\|\vLambda_{\phi}^{-1}\left(\vU\vSigma_{\phi}\vU^{\t}+\vLambda_{\phi}^{-1}\right)^{-2}\vLambda_{\phi}^{-1}\right\|_2\right)
		\n\\
	&=&
	2\left\|\vLambda_{\phi}^{-1}\vLambda\vLambda_{\phi}^{-1}\right\|_2\left(1+\left\|\left(\vLambda_{\phi}\left(\vU\vSigma_{\phi}\vU^{\t}+\vLambda_{\phi}^{-1}\right)^{2}\vLambda_{\phi}\right)^{-1}\right\|_2\right)
		\n\\
	&=&
	2\left\|\vLambda_{\phi}^{-1}\vLambda\vLambda_{\phi}^{-1}\right\|_2\left(1+\left\|\left(\vI+\vLambda_{\phi}\left(\vU\vSigma_{\phi}\vU^{\t}\vLambda_{\phi}^{-1}+\vLambda_{\phi}^{-1}\vU\vSigma_{\phi}\vU^{\t}+(\vU\vSigma_{\phi}\vU^{\t})^2\right)\vLambda_{\phi}\right)^{-1}\right\|_2\right)
		\n\\
	&\leq&
	4\left\|\vLambda_{\phi}^{-1}\vLambda\vLambda_{\phi}^{-1}\right\|_2
	.	\n
\EE
Note that $\vLambda_{\phi}$ and $\vLambda$ are both diagonal. Hence, we have 
\[
	\|\vLambda_{\phi}^{-1}\vLambda\vLambda_{\phi}^{-1}\|_2\=\max_{1\leq i\leq n}\frac{\Lambda_{i,i}}{\left(\Lambda_{i,i}+\frac{\phi_n\mu_n}{2}\right)^2}\leq \frac{1}{\Lambda_{n,n}}.
\]
To lower bound $\Lambda_{n,n}$, we use the following \namecref{lem:eigenprevious}.
\begin{lemma}[Lemma 10 of \citealp{xu2019consistent}] \label{lem:eigenprevious}
Let $\vX\in \bbR^{n\times p}$ be a standard Gaussian random matrix, and let $\vx_i$ be the $i$-th row of the matrix $\vX$. Let $\rho=n/p>1$. There exist constants $c,c'>0$ such that for large enough $n$, with probability at least $1-c'(p^2+n^2)e^{-cn}$, the eigenvalues of
$\frac{1}{n}\vX^{\t}\vX$ and of
  $\frac{1}{n}(\vX^{\t}\vX-\vx_i\vx_i^{\t})$ for each $i=1,\dotsc,n$ are contained in the interval
  \[ \intoo{ \frac{1}{2}\cdot\min\cbr{(1-1/\sqrt{\rho})^2,1/\rho}, 9\rho^2 } . \]
\end{lemma}
Hence, by Lemma \ref{lem:eigenprevious}, we have $\Lambda_{n,n}\geq \frac{1}{2}\min((1-\sqrt{\ration/\ratiop})^2,\ration/\ratiop)>0$ hold with probability $1-c\cdot n^2\mathrm{exp}(-c'n)$ for some absolute constants $c,c'>0$. Hence, we have
\BE
	\left\|\vbS_n\left(\frac{1}{n}\vbXP^{\t}\vbXP\right)\vbS_n\right\|_2
	 \ \leq \ \Op(1) \label{eq:eigenpart2}
\EE
as required.
\end{proof}

\begin{proposition}
  \label{prop:p>npart2goal3}
  \[
	\frac{1}{n}\vtx_i^{\t}\left(\vtS_n^{\backslash i}-\mu_n\left(\vtS_n^{\backslash i}\right)^2\right)\vtx_i
	\=
	\frac{1}{n}\trace{\vtSigmaP\vtS_n-\mu_n\vtSigmaP\vtS_n^2}+\Op\left(\frac{\ln N}{\sqrt{N}}\right)
  .
\]
\end{proposition}
\begin{proof}
It is clear that we just need to prove the following two arguments
\BE
	\sup_i\left|\frac{1}{n}\vtx_i^{\t}\left(\vtS_n^{\backslash i}-\mu_n\left(\vtS_n^{\backslash i}\right)^2\right)\vtx_i
	-
	\frac{1}{n}\trace{\vtSigmaP\vtS_n^{\backslash i}-\mu_n\vtSigmaP\left(\vtS_n^{\backslash i}\right)^2}\right|	
	&=&
	\Op\left(\frac{\ln N}{\sqrt{N}}\right)\n\\ \label{eq:twoarguments1} \\
	\sup_i\left|\frac{1}{n}\trace{\vtSigmaP\vtS_n-\mu_n\vtSigmaP\left(\vtS_n\right)^2}
	-
	\frac{1}{n}\trace{\vtSigmaP\vtS_n^{\backslash i}-\mu_n\vtSigmaP\left(\vtS_n^{\backslash i}\right)^2}\right|
	&=&
	\Op\left(\frac{\ln N}{\sqrt{N}}\right).\n\\	\label{eq:twoarguments2}
\EE
To show \eqref{eq:twoarguments1}, we use a proof similar to that of \eqref{eq:eigenpart2}.
By \Cref{lem:eigenprevious}, we know
\BE
	\max_i\left\|\vbS_n^{\backslash i}\left(\frac{1}{n}\vbXP^{\t}\vbXP-\frac{1}{n}\vbx_i\vbx_i^{\t}\right)\vbS_n^{\backslash i}\right\|_2 &=& \Op(1).	\label{eq:eigenpart2_eq2}
\EE
Note that 
\BE
	\vtx_i^{\t}\left(\vtS_n^{\backslash i}-\mu_n\left(\vtS_n^{\backslash i}\right)^2\right)\vtx_i
	&=&
	\vbx_i^{\t}\vtSigmaP^{1/2}\left(\vtS_n^{\backslash i}-\mu_n\left(\vtS_n^{\backslash i}\right)^2\right)\vtSigmaP^{1/2}\vbx_i
		\n\\
	&=&
	\vbx_i^{\t}\left(\vbS_n^{\backslash i}\left(\frac{1}{n}\vbXP^{\t}\vbXP-\frac{1}{n}\vbx_i\vbx_i^{\t}\right)\vbS_n^{\backslash i}\right)\vbx_i
	.\n
\EE
Furthermore, $\vtx_i$ is a standard Gaussian vector, and it is independent of the matrix
\[ \vbS_n^{\backslash i}\left(\frac{1}{n}\vbXP^{\t}\vbXP-\frac{1}{n}\vbx_i\vbx_i^{\t}\right)\vbS_n^{\backslash i} . \]
Hence, we apply the same proof of \eqref{eq:similarpart1} with \eqref{eq:eigenpart2_eq2} and Lemma 1 of \cite{laurent2000adaptive}; this gives
\begin{multline*}
	\sup_i\left|\frac{1}{n}\vbx_i^{\t}\vtSigmaP^{1/2}\left(\vtS_n^{\backslash i}-\mu_n\left(\vtS_n^{\backslash i}\right)^2\right)\vtSigmaP^{1/2}\vbx_i-
	\frac{1}{n}\trace{\vtSigmaP^{1/2}\left(\vtS_n^{\backslash i}-\mu_n\left(\vtS_n^{\backslash i}\right)^2\right)\vtSigmaP^{1/2}}\right| \\
  = \Op\left(\frac{\ln N}{\sqrt{N}}\right)
	.
\end{multline*}
Hence, \eqref{eq:twoarguments1} holds.
Therefore, it remains to show \eqref{eq:twoarguments2}, which is equivalent to
\BE
	\sup_i\left|\frac{1}{n}\trace{\vbS_n\left(\frac{1}{n}\vbXP^{\t}\vbXP\right)\vbS_n}
	-
	\frac{1}{n}\trace{\vbS_n^{\backslash i}\left(\frac{1}{n}\vbXP^{\t}\vbXP-\frac{1}{n}\vbx_i\vbx_i^{\t}\right)\vbS_n^{\backslash i}}\right|&=&\Op\left(\frac{\ln N}{\sqrt{N}}\right).	\n\\
	\label{eq:finalgoal}
\EE
By the Sherman-Morrison formula, we have 
\[
	\vbS_n\= \vbS_n^{\backslash i}-\frac{\vbS_n^{\backslash i}\vbx_i\vbx_i^{\t}\vbS_n^{\backslash i}}{n+\vbx_i^{\t}\vbS_n^{\backslash i}\vbx_i},
\]
and therefore
\BE
	\lefteqn{\trace{\vbS_n\left(\frac{1}{n}\vbXP^{\t}\vbXP\right)\vbS_n-\vbS_n^{\backslash i}\left(\frac{1}{n}\vbXP^{\t}\vbXP-\frac{1}{n}\vbx_i\vbx_i^{\t}\right)\vbS_n^{\backslash i}}}
		\n\\
	&=&
	 \trace{\vbS_n^{\backslash i}\left(\frac{1}{n}\vbx_i\vbx_i^{\t}\right)\vbS_n^{\backslash i}}-2\cdot\trace{\vbS_n^{\backslash i}\left(\frac{1}{n}\vbXP^{\t}\vbXP\right)\frac{\vbS_n^{\backslash i}\vbx_i\vbx_i^{\t}\vbS_n^{\backslash i}}{n+\vbx_i^{\t}\vbS_n^{\backslash i}\vbx_i}}
	 	\n\\
	&&+\trace{\frac{\vbS_n^{\backslash i}\vbx_i\vbx_i^{\t}\vbS_n^{\backslash i}}{n+\vbx_i^{\t}\vbS_n^{\backslash i}\vbx_i}\left(\frac{1}{n}\vbXP^{\t}\vbXP\right)\frac{\vbS_n^{\backslash i}\vbx_i\vbx_i^{\t}\vbS_n^{\backslash i}}{n+\vbx_i^{\t}\vbS_n^{\backslash i}\vbx_i}}
	.	\label{eq:p>npart2final_eq1}
\EE
Let $\vM_i=\vbS_n^{\backslash i}\left(\frac{1}{n}\vbXP^{\t}\vbXP-\frac{1}{n}\vbx_i\vbx_i^{\t}\right)\vbS_n^{\backslash i}$. Let $\rho = \frac{\mu_n}{n}\vbx_i^{\t}\vbS_n^{\backslash i}\vbx_i$ and $\tau = \frac{\mu_n^2}{n}\vbx_i^{\t}\left(\vbS_n^{\backslash i}\right)^2\vbx_i$. Then, from \eqref{eq:p>npart2final_eq1}, we have
\BE
	\lefteqn{\sup_i\left|\frac{1}{n}\trace{\vbS_n\left(\frac{1}{n}\vbXP^{\t}\vbXP\right)\vbS_n-\vbS_n^{\backslash i}\left(\frac{1}{n}\vbXP^{\t}\vbXP-\frac{1}{n}\vbx_i\vbx_i^{\t}\right)\vbS_n^{\backslash i}}\right|}
		\n\\
	&=&
	\sup_i\left|\frac{\tau}{\mu_n^2n}-\frac{2}{n}\cdot\left(\trace{\vM_i\frac{\frac{\mu_n}{n}\vbx_i\vbx_i^{\t}\vbS_n^{\backslash i}}{\mu_n+\rho}}+\frac{\rho}{\mu_n+\rho}\frac{\tau}{\mu_n^2}\right)+\frac{1}{n}\frac{\tau}{(\mu_n+\rho)^2}\cdot\frac{1}{n}\vbx_i^{\t}\vM_i\vbx_i+\frac{1}{\mu_n^2n}\frac{\rho^2\tau}{(\mu_n+\rho)^2}\right|
		\n\\
	&=&
	\sup_i\left|\frac{\left((\mu_n+\rho)^2-2\rho(\mu_n+\rho)+\rho^2\right)\tau}{\mu_n^2n(\mu_n+\rho)^2}-\frac{2}{n}\cdot\trace{\vM_i\frac{\frac{\mu_n}{n}\vbx_i\vbx_i^{\t}\vbS_n^{\backslash i}}{\mu_n+\rho}}+\frac{1}{n}\frac{\tau}{(\mu_n+\rho)^2}\cdot\frac{1}{n}\vbx_i^{\t}\vM_i\vbx_i\right|
		\n\\
	&\leq&
	\sup_i\left(\frac{\tau}{n(\mu_n+\rho)^2}+\frac{2}{n(\mu_n+\rho)}\|\vM_i\|_2\cdot\frac{\mu_n}{n}\|\vbx_i\vbx_i^{\t}\vbS_n^{\backslash i}\|_2+\frac{\tau}{n(\mu_n+\rho)^2}\cdot \frac{1}{n}\|\vbx_i\|_2^2\|\vM_i\|_2\right)
		\n\\
	&\leq&
		\sup_i\left(\frac{\tau}{n(\mu_n+\rho)^2}+\frac{2}{n(\mu_n+\rho)}\|\vM_i\|_2\cdot\sqrt{\frac{1}{n}\|\vbx_i\|_2^2\cdot\tau}+\frac{\tau}{n(\mu_n+\rho)^2}\cdot \frac{1}{n}\|\vbx_i\|_2^2\|\vM_i\|_2\right)
	.	\label{eq:p>npart2final_eq2}
\EE 
Our next step is to bound $\rho,\tau,\sup_i\|\vbx_i\|_2$ and $\sup_i\|\vM_i\|_2$.
Since the $\vbx_i$ are standard Gaussian vectors, standard $\chi^2$ tail bounds \citep{laurent2000adaptive} establish that $\sup_i\frac{1}{n}\|\vx_i\|_2^2=\Op(\ln N)$.
Then, by \eqref{eq:maximum}, we know 
\[
	\tau \=  \frac{\mu_n^2}{n}\vbx_i^{\t}\left(\vbS_n^{\backslash i}\right)^2\vbx_i \leq \mu_n^2\cdot \Op(\ln N)\cdot \Op\left(\frac{1}{\mu_n^2}\right)\=\Op\left(\ln N\right).
\]
Using \Cref{claim:p>ngoal4}, \Cref{claim:p>ngoal5}, and \eqref{eq:36}, we also have $\rho = \Thp(1)$.
Finally, by \eqref{eq:eigenpart2}, we have $\sup_i\|\vM_i\|=\Op(1)$. Plug in these results in \eqref{eq:p>npart2final_eq2}, we have
\[
	\sup_i\left|\frac{1}{n}\trace{\vbS_n\left(\frac{1}{n}\vbXP^{\t}\vbXP\right)\vbS_n-\vbS_n^{\backslash i}\left(\frac{1}{n}\vbXP^{\t}\vbXP-\frac{1}{n}\vbx_i\vbx_i^{\t}\right)\vbS_n^{\backslash i}}\right|
	\=
	\Op\left(\frac{\ln^2 N}{N}\right)
	.
\]
Hence \eqref{eq:finalgoal} holds.
\end{proof}

\subsection{Proof of \Cref{lem:St}}\label{app:St}

The first part of the lemma, Equation~\eqref{eq:lem_m}, follows from Theorem 2.38 of \cite{tulino2004random}.

For the second part, to lower bound the minimum eigenvalue $\lambda_{\min}$ of $\frac{1}{n}\vbX\vH\vbX^{\t}$, we need to find the support of $\mathcal{F}$. From Section 4 of \cite{silverstein1995analysis}, we have 
\BE
	z\in ~\operatorname{supp}(\mathcal{F})^c &\Leftrightarrow& m(z)\in B \ \ \text{and}\ \ \frac{1}{m(z)^2}-\gamma \int_{\eta_1}^{\infty} \frac{t^2f_h(t)\dif t}{(1+tm(z))^2}>0,	\n
\EE
where $B:=\{m: m\neq 0, -m^{-1}\in \operatorname{supp}(\mathcal{H})^c \}$.

To show $\lambda_{\min} > c_{\epsilon}>0$ holds in probability for some small enough constant $c_{\epsilon}$, we just need to show that for all $0\leq z\leq c_{\epsilon}$, 
\BE
	m(z)\ >\ 0 \quad \text{and} \quad \frac{1}{m(z)^2}-\gamma \int_{\eta_1}^{\infty} \frac{t^2}{(1+t\cdot m(z))^2}\cdot f_h(t)\dif t >0
	.	\label{eq:support}	
\EE   
Note that the equation \eqref{eq:lem_m} defining $m(z)$, i.e.,
\[
	m(z) \= -\left(z-\gamma \int_{\eta_1}^{\infty} \frac{t f_h(t)\dif t}{1+t\cdot m(z)}\right)^{-1} , \quad \forall z\in \operatorname{supp}(\mathcal{F})^c
\]
is equivalent to
\[
	z\=\gamma \int_{\eta_1}^{\infty}\frac{t}{1+t\cdot m(z)}\cdot f_h(t)\dif t-\frac{1}{m(z)}, \quad \forall z\in ~\operatorname{supp}(\mathcal{F})^c
\]
Let us consider the ``inverse'' of $m(z)$ defined by the following equation:
\[
	z(m) \ := \ \gamma \int_{\eta_1}^{\infty}\frac{t}{1+t\cdot m}\cdot f_h(t)\dif t-\frac{1}{m}.
\]
Note that 
\BE
	\inf_{m<0}z(m)
	&\geq& \gamma \ >\ 1.\n
\EE
Hence, for all $z\leq 1$, if $m(z)$ exists, we have $m(z)>0$. Further, note that
\BE
	\frac{\dif z(m)}{\dif m}>0&\Leftrightarrow&\frac{1}{m(z)^2}-\gamma \int_{\eta_1}^{\infty} \frac{t^2}{(1+t\cdot m(z))^2}\cdot f_h(t)\dif t >0
		\n\\
	&\Leftrightarrow&
	\gamma \int_{\eta_1}^{\infty} \frac{t^2}{(m^{-1}+t)^2}\cdot f_h(t)\dif t<1
	.	\n
\EE
Moreover, $\gamma\int_{\eta_1}^{\infty} \frac{t^2}{(m^{-1}+t)^2}\cdot f_h(t)\dif t$ is a continuous increasing function of $m$ with
\BE
	\gamma\int_{\eta_1}^{\infty} \frac{t^2}{(m^{-1}+t)^2}\cdot f_h(t)\dif t&\rightarrow& 0 \quad \text{as} \ m\rightarrow 0
	\n\\
	\gamma\int_{\eta_1}^{\infty} \frac{t^2}{(m^{-1}+t)^2}\cdot f_h(t)\dif t&\rightarrow& \gamma>1 \quad \text{as} \ m\rightarrow \infty
	.	\n
\EE
Therefore, we know there exists a constant $m_c$ such that for all $0<m<m_c$, $z(m)$ is a strictly increasing function on $m\in (0,m_c)$ and strictly decreasing function on $m\in [m_c,\infty)$. Thus, the conditions in \eqref{eq:support} (with $m$ in place of $m(z)$) are met for all $0<m<m_c$. Note that 
\BE
	m \cdot z(m)&=&\left(\gamma \int_{\eta_1}^{\infty} \frac{t}{1/m+t}\cdot f_h(t)\dif t-1\right)
	\ \rightarrow\ 
	\left\{\begin{aligned}
	&-1,  &&\text{as}\ m\rightarrow 0^{+}\\
	&\gamma-1>0, && \text{as}\ m\rightarrow +\infty
	\end{aligned}\right.
	,	\n
\EE
Therefore, we have $z(m)\rightarrow -\infty$ as $m\rightarrow 0^{+}$ and $z(m)\rightarrow 0^{+}$ as $m\rightarrow \infty$. Then, by continuity of the function $z(m)$, we know for any non-positive value $z$, the mapping between $z$ and $m>0$ defined by \eqref{eq:lem_m} is an one to one mapping. Moreover, since the function $z(m)$ is increasing on $\intoo{0,m_c}$ and decreasing on $\intco{m_c,\infty}$, there exists an unique $m^*$ such that $z(m^*)=0$ and $z(m)$ is a continuous and increasing function on $[0,m^*]$. Hence, we have $m^*<m_c$. This implies $m(z)$ is a continuous increasing function on $z\leq 0$. Further, we can find a small enough constant $\epsilon>0$ such that $m^*+\epsilon<m_c$ and $0<z(m^*)<1$ ($z$ is a function here). With $c_{\epsilon} := z(m^*+\epsilon)$, we have that for all $0\leq z\leq c_{\epsilon}$, the conditions in \eqref{eq:support} are met. Hence $\lambda_{\min} >c_{\epsilon}>0$ holds in probability. 

Finally, by the dominated convergence theorem, we have
\BE
	\lim_{n\rightarrow \infty}m_n(z)= m(z), \ \text{a.s.} \quad \text{and} \quad \lim_{n\rightarrow \infty}m_n'(z)=m'(z), \ \text{a.s. \ for}\quad \forall z<0.\n
\EE 
For an increasing sequence $z_n\rightarrow 0^{-}$, note that for all $\epsilon'>0$, we have $|m_n(z_n)-m_n(-\epsilon')|\leq \frac{\epsilon'-z_n}{c_{\epsilon}^2}$ holds in probability. Further, $m_n(-\epsilon')\rightarrow m(-\epsilon')$ almost surely and $m(-\epsilon')\rightarrow m(0)$ as $\epsilon'\rightarrow 0$. Hence, for all $\epsilon'>0$, we can choose a small enough $\epsilon''>0$ such that
\BE
	\P(|m_n(z_n)-m_n(-\epsilon'')| \leq  \frac{\epsilon'}{3})
	&\rightarrow& 1
	\n\\
	\P(|m_n(-\epsilon'')-m(-\epsilon'')|\leq \frac{\epsilon'}{3})&\rightarrow& 1
	\n\\
	|m(-\epsilon'')-m(0)|&\leq& \frac{\epsilon'}{3}
	.	\n
\EE
Hence, we have $m_n(z_n)\ipt m(0)$. Similarly, we have $m_n'(z_n)\ipt m'(0)$.

\subsection{Proof of \Cref{lem:dH}}\label{app:dH}
Let $\sigma_n$ be the random variable that follows the empirical eigenvalue distribution of $N^{\kappa}\vSigma_S$. Since the minimum eigenvalue of $N^{\kappa}\vSigma_S$ is $\frac{N^{\kappa}}{p_2^{\kappa}}$ and its maximum eigenvalue is $\frac{N^{\kappa}}{(p_1+1)^{\kappa}}$. Then for all $t\in [\frac{N^{\kappa}}{p_2^{\kappa}},\frac{N^{\kappa}}{(p_1+1)^{\kappa}}]$, we have
\BE
	\P(\sigma_n>t)
	&=&
	\frac{1}{|S|}\sum_{i=1+p_1}^{p_2}\ind{\frac{N^{\kappa}}{i^{\kappa}}>t}
		\n\\
	&=&
	\frac{1}{|S|}\max\left(0,\left\lfloor\frac{N}{t^{1/\kappa}}\right\rfloor-p_1\right)
		\n\\
	&=&
	\frac{1}{|S|}\left(\left\lfloor\frac{N}{t^{1/\kappa}}\right\rfloor-p_1\right)
	,	\n
\EE 
where the last inequality is due to the fact that 
\[
	\left\lfloor\frac{N}{t^{1/\kappa}}\right\rfloor
	\ \geq \
	\left\lfloor\frac{N(p_1+1)}{N}\right\rfloor
	\=
	\left\lfloor p_1+1\right\rfloor
	\ \geq \
	p_1
	.
\]
Hence, as $N\rightarrow \infty$, we have
\[
	\P(\sigma_n>t)\rightarrow \left\{\begin{aligned}
	&1,	&&t\leq \frac{1}{\alpha_2^{\kappa}}\\
	&\max\left(0,\frac{1}{\alpha_2-\alpha_1}(\frac{1}{t^{1/\kappa}}-\alpha_1)\right), && t>\frac{1}{\alpha_2^{\kappa}}
	\end{aligned}\right.
	.
\]
Hence, the probability density function for the limiting distribution of $\sigma_n$ is indeed $f(s)$ given by \eqref{eq:pdf}.

\subsection{Proof of \Cref{lem:eigen}}\label{app:eigen}
Without loss of generality, we assume that the diagonal elements of $\vSigma$ are in a non-increasing order. We condition on the event where the $\frac{n}{2}$ smallest diagonal elements of $\vSigma$ are lower-bounded by $\nu$.
The minimum eigenvalue of
\[ \vS=\left(\frac{1}{n}\vbX^{\t}\vbX+\mu \vSigma\right) , \]
is given by
\BE
	\sigma_{\min}(S)&=&\min_{\|\vv\|=1}\vv^{\t}\left(\frac{1}{n}\vbX^{\t}\vbX+\mu \vSigma\right)\vv.	\n
\EE
Let $\vv=(\vv_1,\vv_2)$ where $\vv_1$ is the first $p-\frac{n}{2}$ number of components of $\vv$ and $\vv_2$ is the last $\frac{n}{2}$ number of components of $\vv$. If $\|\vv_1\|^2\geq \frac{1}{400\gamma^2}$, then immediately, we have
\BE
	\sigma_{\min}(S)&\geq&\mu \|\vv_1\|^2\nu\ \geq \ \mu \frac{\nu}{400\gamma^2}.\n
\EE  
Otherwise, let $\vbX=(\vbX_1,\vbX_2)$ where $\vbX_1$ is the first $p-\frac{n}{2}$ columns of $\vbX$ and $\vbX_2$ is the last $\frac{n}{2}$ columns of $\vbX$. Then we have
\BE
	\sigma_{\min}(S)&\geq &\min_{\|\vv\|=1,\|\vv_1\|^2<\frac{1}{400\gamma^2}} \ \frac{1}{n}\|\vbX_2\vv_2\|^2+\frac{1}{n}\|\vbX_1\vv_1\|^2-2\frac{1}{n}\|\vbX_1\vv_1\|\cdot\|\vbX_2\vv_2\|
		\n\\
	&=&
	\min_{\|\vv\|=1,\|\vv_1\|^2<\frac{1}{400\gamma^2}} \ \left(\frac{1}{\sqrt{n}}\|\vbX_2\vv_2\|-\frac{1}{\sqrt{n}}\|\vbX_1\vv_1\|\right)^2.	\n
\EE
Note that $\vX_2$ is a $n\times \frac{n}{2}$ standard Gaussian matrix and therefore the minimum eigenvalue of $\frac{1}{n}\vX_2^{\t}\vX_2$ can be lower bounded away from 0. Further $\vX_1$ is a $n\times (p-\frac{n}{2})$ standard Gaussian matrix with $\frac{p-\frac{n}{2}}{n}\rightarrow \gamma-\frac{1}{2}$ as $p,n\rightarrow \infty$. Hence the maximum eigenvalue of $\frac{1}{n}\vX_1^{\t}\vX_1$ can be upper bounded. In fact, from Lemma \ref{lem:eigenprevious} (Lemma 10 of \cite{xu2019consistent}), we have with probability $1-c n^2\mathrm{exp}(-c'n)$, we have
\[
	\min_{\|\vv\|=1}\frac{1}{n}\|\vbX_2\vv\|^2\geq \frac{1}{25} \quad \text{and} \quad \max_{\|\vv\|=1}\frac{1}{n}\|\vbX_1\vv\|^2\leq 9\gamma^2.
\]
Hence, we have
\BE
	\sqrt{\sigma_{\min}(S)}&\geq&\min_{\|\vv\|=1,\|\vv_1\|^2<\frac{1}{400\gamma^2}} \frac{1}{\sqrt{n}}\|\vbX_2\vv_2\|-\frac{1}{\sqrt{n}}\|\vbX_1\vv_1\|
		\n\\
	& \geq &
	\frac{1}{5}\sqrt{1-\frac{1}{400\gamma^2}}-3\gamma \cdot \frac{1}{20\gamma}
		\n\\
	&\geq&
	\frac{\sqrt{399}}{100}-\frac{3}{20}
	\ >\ 0. \n
\EE
This completes the proof of this lemma.

\section{Analysis under polynomial eigenvalue decay with noise $\sigma > 0$}\label{app:noise}

In this section, we consider analogues of \Cref{thm:p<n}--\Cref{thm:compare} that permit noisy independent observations
\[ y_i = \vx_i^\t\vtheta + w_i , \quad i=1,\dotsc,n , \]
where $\vw = (w_1,\dotsc,w_n) \sim \Normal(\v0,\sigma^2\vI)$, where we allow $\sigma^2>0$.
\begin{theorem}\label{thm:noise}
  Assume \A1 with constant $\kappa$ and \A2 with constants $\ratiop$ and $\ration$.
\begin{itemize}
\item[(i)]
We have for all $\ratiop <\ration $,
\BE
\bbE_{\vw,\vtheta}[\error] \ \ipt \ \left( N^{1-\kappa}\int_{\ratiop}^1t^{-\kappa}\dif t+\sigma^2\right)\cdot \frac{\ration }{\ration -\ratiop }=:\mathcal{R}_{\kappa}(\ratiop,\sigma), \quad \forall \ratiop<\ration. \label{eq:iptnoise}
\EE
When $\kappa>1$, the minimum of $\mathcal{R}_{\kappa}(\ratiop, \sigma)$ is achieved at $\ratiop=0$ and the minimum risk is given by
\BE
	\min_{\ratiop <\ration }\ \mathcal{R}_{\kappa}(\ratiop, \sigma)
	&=&
	\sigma^2
	.\label{eq:minp<nnoise_eq1}
\EE 
When $\kappa\leq1$, we have nearly the same results as in \Cref{thm:p<n}, i.e., the minimum of $\mathcal{R}_{\kappa}(\ratiop, \sigma)$ is achieved at $\ratiop^*$ which is the unique solution of the equation $h_{\kappa}(\ratiop )=0$ on $(0,\ration)$, where $h_{\kappa}(\ratiop)$ is given by
\BE
	h_{\kappa}(\ratiop )
	&:=&
	\frac{\ration }{\ratiop }-\int_{\ratiop }^1t^{\kappa-2}\dif t-1-\sigma^2\ind{\kappa=1}
	.\label{eq:hnoise_eq1}
\EE	
The minimum risk is therefore given by
\BE
	\min_{\ratiop <\ration }\ \mathcal{R}_{\kappa}(\ratiop, \sigma)
	&=&
	N^{1-\kappa}\frac{\ration }{(\ratiop^*)^{\kappa}}
	.\label{eq:minp<nnoise_eq2}
\EE
\item[(ii)]
For all $\ratiop >\ration$,
the function $m_{\kappa}$ defined in \Cref{eq:mz} and its derivative $m_{\kappa}'$ are well-defined and positive at $z=0$, and
\BE
\bbE_{\vw,\vtheta}[\error] \ \Ipt \ 
	N^{1-\kappa}\frac{\ration }{m_{\kappa}(0)}+\left(N^{1-\kappa}\int_{\ratiop}^{1}t^{-\kappa}\dif t+\sigma^2\right)\frac{m_{\kappa}'(0)}{m^2_{\kappa}(0)}=:\mathcal{R}_{\kappa}(\ratiop,\sigma) . \label{eq:iptp>nnoise}
\EE

\item[(iii)] 
When $\kappa>1$, the minimum risk for all $\ratiop<1$ and $\ratiop \neq \ration$ is achieved at $\ratiop=0$, i.e., $p=o(n)$. When $\kappa<1$, let $\ratiop^*$ be the minimizer of $\mathcal{R}_{\kappa}(\ratiop, \sigma)$ over the interval $\intco{0,\ration}$. Then $\limsup_N\mathcal{R}_{\kappa}(1, \sigma) / \mathcal{R}_{\kappa}(\ratiop^*, \sigma) < 1$.

\end{itemize}
\end{theorem}

The proof of (i) can be easily derived from \eqref{eq:ipta>0}. The proof of (ii) can be easily derived as well from \eqref{eq:bbEwthetaerrorp>n} and \eqref{eq:p>npart2goal}. For the proof of (iii), note that when $\kappa<1$, the dominant part of the risk is the same as the noiseless case, so (iii) follows from the arguments in \Cref{thm:compare}. When $\kappa>1$, the dominant part of the risk is the noise, and therefore from \eqref{eq:m**}, we have
\[
	\min_{\ratiop >\ration }\ \mathcal{R}_{\kappa}(\ratiop, \sigma)\ \geq \ \min_{\ratiop >\ration }\sigma^2\frac{\ration(1+(s^*_{\kappa})^{\kappa})}{\ration+(\ration-\ratiop)(s^*_{\kappa})^{\kappa}}\ > \ \sigma^2\ =\ \mathcal{R}_{\kappa}(0, \sigma).
\] 
Further (and still with $\kappa>1$), for $N$ large enough,
\[
    \min_{\ratiop<\ration} \ \mathcal{R}_{\kappa}(\ratiop, \sigma)\ \rightarrow \ \min_{\ratiop<\ration} \ \frac{\ration}{\ration-\ratiop} \sigma^2 \ \geq \ \sigma^2\ = \  \mathcal{R}_{\kappa}(0, \sigma) .
\] 
This proves (iii) in the case $\kappa>1$.

\section{Proof of \Cref{thm:general}}\label{app:general}

\subsection{Proof of Part (i)}
Since $p<n$ holds almost surely as $N\rightarrow \infty$, by excluding an additional zero probability event $p\geq n$, we can apply the same calculation in Section \ref{sec:proofp<n} and conclude that the following equation holds under our new settings, i.e.,
\BE
\bbE_{\vw,\vtheta}[\error] \ \ipt \ \left( \trace{\vSigmapC}+\sigma^2\right)\frac{\ration}{\ration-\ratiop(\nu)}. \n
\EE
Hence, to show \eqref{eq:generalp<nthm_1}, we just need to characterize $\trace{\vSigmapC}$. By Assumption $\B1$, we have
\BE
	\trace{\vSigmapC}\=\left(\sum_{i=1}^N\sigma_i^2\ind{c_N\sigma_i^2\leq \nu}\right)\rightarrow \frac{N}{c_N}\cdot\delta\int_{\eta_1}^{\nu}tf(t)\dif t.\n
\EE
Hence, we have
\BE
  \bbE_{\vw,\vtheta}[\error] & \ipt &
	\left(\frac{N}{c_N}\cdot\delta\int_{\eta_1}^{\nu}tf(t)\dif t+\sigma^2\right)\frac{\ration}{\ration-\delta\int_{\nu}^{\infty} f(t)\dif t }. \n
\EE
Hence, \eqref{eq:generalp<nthm_1} holds. Then our next step is to find the optimal $\nu^*$ in $(\nu_b,\infty)$ when $\sigma=0$. Define 
\[
	g_f(\nu)\ := \ \frac{\int_{\eta_1}^{\nu}tf(t)\dif t}{\ration-\delta\int_{\nu}^{\infty} f(t)\dif t}.
\]
To minimize $\bbE_{\vw,\vtheta}[\error]$, we just need to minimize $g_f(\nu)$ over $\nu \in (\nu_b,\infty)\bigcap \operatorname{supp}(f)$. To do this,  we analyze the first derivative of $g_f(\nu)$. Note that
\BE
	\frac{\dif g_f(\nu)}{\dif \nu}
	&=&
	\frac{\nu f(\nu)}{\ration-\delta\int_{\nu}^{\infty} f(t)\dif t }-\frac{\delta f(\nu)\int_{\eta_1}^{\nu}tf(t)\dif t}{\left(\ration-\delta\int_{\nu}^{\infty} f(t)\dif t \right)^2}
		\n\\
	&=&
	\frac{f(\nu)}{\left(\ration-\delta\int_{\nu}^{\infty} f(t)\dif t \right)^2}\left(\nu\ration-\nu\delta\int_{\nu}^{\infty}f(t)\dif t-\delta\int_{\eta_1}^{\nu}tf(t)\dif t\right)
		\n\\
	&=&
	\frac{f(\nu)}{\left(\ration-\delta\int_{\nu}^{\infty} f(t)\dif t \right)^2}h_f(\nu)
	,	\quad \forall \nu \in (\nu_b,\infty)\bigcap \operatorname{supp}(f). \n
\EE
Therefore, the sign of $\frac{\dif g_f(\nu)}{\dif \nu}$ is the same as the sign of $h_f(\nu)$ on $\nu \in (\nu_b,\infty)\bigcap \operatorname{supp}(f)$. Further, note that 
\BE
	\frac{\dif h_f(\nu)}{\dif \nu}&=& \ration-\delta\int_{\nu}^{\infty}f(t)\dif t\ >\ 0, \quad \forall \nu \in (\nu_b,\infty)\bigcap \operatorname{supp}(f). \n
\EE
Hence $h_f(\nu)$ is a strictly increasing function of $\nu$ in $(\nu_b,\infty)\bigcap \operatorname{supp}(f)$. Further, note that
\BE
	\lim_{\nu\rightarrow \nu_b}h_f(\nu)&=&-\delta\int_{\eta_1}^{\nu_b}tf(t) \ < \ 0. \n
\EE
Hence, by continuity of $h_f(\nu)$, either equation $h_f(\nu)=0$ admits an unique solution denoted by $\nu^*$ on $(\nu_b,\infty)\bigcap \operatorname{supp}(f)$ or $h_f(\nu)<0$ holds for all $\nu \in (\nu_b,\infty)\bigcap \operatorname{supp}(f)$. Hence, the minimum risk is achieved at $\nu =\nu^*$ if $\nu^*$ exists. Otherwise, it is achieved at any $\nu \in \bbR\bigcup \{+\infty\}$ such that $\int_\nu^{\infty}f(s)\dif s=0$. Hence, if $\nu^*$ exists, the value of the minimum risk given by
\BE
  \bbE_{\vw,\vtheta}[\error] &\ipt& \frac{N}{c_N}\cdot\frac{\ration}{\ration-\delta\int_{\nu^*}^{\infty} f(t)\dif t }\cdot \delta\int_{\eta_1}^{\nu^*}tf(t)\dif s \=  \frac{N}{c_N}\cdot \ration\nu^*
	, \n
\EE
where the last equation is due to the fact that $h_f(\nu^*)=0$. Otherwise, the value of the minimum risk given by
\BE
\bbE_{\vw,\vtheta}[\error]
	&\ipt&
    \frac{N}{c_N} \delta\int_{\eta_1}^{\infty}tf(t)\dif t.\n
\EE

\subsection{Proof of Part (ii)}
We apply the same strategy for the proof of Theorem \ref{thm:p>n}.  Since the proof is similar to the proof we have shown for Theorem \ref{thm:p>n} in Section \ref{sec:proofp>n} and Appendix \ref{app:p>n}, we only address a few differences here. 

From Section \ref{sec:proofp>n}, we should first show that equation $q_f(s,\nu)=0$ admits an unique solution on $(0,\infty)$. Note that
\BE
	\frac{\partial q_f(s,\nu)/s}{\partial s} &=& \delta\int_{\nu}^{\infty} \frac{tf(t)}{(s+t)^2}\dif t \ >\ 0. \label{eq:generalrelation1} 
\EE
Hence, $q_f(s,\nu)/s$ is a strictly increasing function of $s$ on $s\in(0,\infty)$. Further, since $\nu<\nu_b$, we have
\BE
	\lim_{s\rightarrow 0}\frac{q_f(s,\nu)}{s}&=&\ration-\delta\int_{\nu}^{\infty}f(t)\dif t \=  \ < \ 0,	\n\\
	\lim_{s\rightarrow \infty}\frac{q_f(s,\nu)}{s}&=&\ration - 0 \ > 0.	\label{eq:generalrelation_eq2}
 \EE	
Hence, by continuity of function $q_f(s,\nu)/s$, we know $q_f(s,\nu)/s=0$ admits a unique solution denoted by $s^*_f$ on $(0,\infty)$.

Note that with the same proof shown in Section \ref{sec:proofp>n}, we have
\BE
\bbE_{\vw,\vtheta}[\error]
	&=&	
	\left(\underbrace{\trace{\vSigmaP\left(\vI-\vP^{\perp}_{\vXP}\right)}}_{\text{part 1}}+\underbrace{\trace{\vXpC^{\t}\left(\vXP\vXP^{\t}\right)^{-1}\vXP\vSigmaP\vXP^{\t}\left(\vXP\vXP^{\t}\right)^{-1}\vXpC}+\trace{\vSigmapC}}_{\text{part 2}}\right)	\n\\
	&&+\sigma^2\underbrace{\left(\trace{\left(\vXP\vXP^{\t}\right)^{-1}\vXP\vSigmaP\vXP^{\t}\left(\vXP\vXP^{\t}\right)^{-1}}+1\right)}_{\text{part 3}}.\n
\EE 

To calculate part 1, we employ the proof strategy shown in Appendix \ref{app:part1} with the following remarks. First, the expression for $\ratiop$ is now given by 
\[
	\ratiop(\nu)=\int_{\nu}^{\infty}f(t)\dif t.
\]
Second, we should choose $\mu_n=\min(\frac{1}{\sqrt{N}},o(1/c_N))$ instead of $\mu_n=\min(\frac{1}{\sqrt{N}},o(N^{-\kappa}))$. Third, to directly apply \Cref{lem:St}, we require $\delta=1$ from Assumption $\B1$. Yet, since we restrict $\ration<\delta$ in Assumption $\B2$, it is straightforward to extend the results in Lemma \ref{lem:St} to handle the case where $\delta\in (0,1)$ by following the proof presented in \Cref{app:St}. The results of \Cref{lem:dH} is directly assumed by Assumption $\B1$. Finally to apply \Cref{lem:eigen}, we require $\frac{n}{2}$ smallest eigenvalue of $(c_N\vSigmaP)^{-1}$ is lower bounded by a positive constant. This can be easily verified due to Assumption $\B1$ and the restriction on $\ration<\delta$. Hence, follow the proof in Appendix \ref{app:part1} with these remarks, we can conclude that 
\BE
	\text{part 1}
	&\ipt&
	\frac{N}{c_N} \cdot \frac{\ration}{m_f(0)}
	,	\label{eq:generalriskp=Npre}
\EE
where $m_f(-\mu)$, the Stieltjes transform of the limiting spectral distribution of the matrix $\frac{1}{n}\vtX\vtX^{\t}$, is given by
\BE
	\mu
	&=&
	\frac{1}{m_f(-\mu)}-\frac{\ratiop(\nu)}{\ration} \cdot \frac{\int_{\nu}^{\infty} \frac{t f(t)}{1+t\cdot m_f(-\mu)}\dif t}{\int_{\nu}^{\infty}f(t)\dif t}
	,	\n
\EE
which is equivalent to
\BE
	\mu
	&=&
	\frac{1}{m_f(-\mu)}-\frac{\delta}{\ration} \cdot \int_{\nu}^{\infty} \frac{t f(t)}{1+t\cdot m_f(-\mu)}\dif t
	.	\label{eq:mmgeneral}
\EE
Therefore, we know $m^*_f=m_f(0)>0$ is the solution of the following equation
\BE
	0&=&
	\frac{\ration}{m^*_f}- \frac{\delta}{m^*_f} \int_{\nu}^{\infty} \frac{t f(t)}{1/m^*_f+t }\dif t
	\= q_f\left(\frac{1}{m^*_f},\nu\right)
	.	\label{eq:mmgeneral2}
\EE
Then $s=\frac{1}{m^*_f}$ should be the solution of equation $q_f(s,\nu)=0$. By uniqueness of $s^*_f$, we have $s^*_f=\frac{1}{m^*_f}$.

For part 2 and part 3, we employ the proof strategy shown in \Cref{app:part2} with a few remarks. First, note that due to Assumption $\B1$, we have
\[
	\trace{c_N\vSigmapC}\rightarrow N \cdot \delta\int_{\eta_1}^{\nu}tf(t)\dif t \quad \text{and }\quad \trace{c_N^2\vSigmapC^2}\rightarrow N \cdot \delta\int_{\eta_1}^{\nu}t^2f(t)\dif t.
\]
Hence, we have the following analogue of \eqref{eq:p>npart2}:
\[
	\text{part 2}\ \ipt \ \frac{N}{c_N} \cdot \delta\int_{\eta_1}^{\nu}tf(t)\dif t \cdot (\psi+1)+\Op\left(\frac{\sqrt{N}}{c_N}\cdot \psi\int_{\eta_1}^{\nu}t^2f(t)\dif t\right),
\]
where $\psi =\trace{\vSigmaP\vXP^{\t}\left(\vXP\vXP^{\t}\right)^{-2}\vXP}$. Finally, to show \eqref{eq:p>npart2goal}, we should choose $\mu_n=\min(\frac{1}{\sqrt{N}},o(1/c_N))$ instead of $\mu_n=\min(\frac{1}{\sqrt{N}},o(N^{-\kappa}))$. Thus, with these remarks and modifications, we can show that
\[
	\text{part 2}\ \ipt \ \frac{N}{c_N} \cdot \delta\int_{\eta_1}^{\nu}tf(t)\dif t \cdot \frac{m_f'(0)}{m^2_f(0)},
\]
and
\[
	\text{part 3}\ \ipt \  \frac{m_f'(0)}{m^2_f(0)}.
\]
Hence, our last step is to characterize $m_f'(0)$ using the chain rule. Note that from \eqref{eq:mmgeneral} and \eqref{eq:mmgeneral2}, we have
\[
	-\ration z = q_f\left(\frac{1}{m_f(z)},\nu\right)
\]
Hence, taking the derivative with respect to $z$ on both sides and with the chain rule, we have
\BE
	-\ration&=&\frac{\partial q_f(s,\nu)}{\partial s}\Big|_{s=\frac{1}{m_f(z)}}\cdot \left(-\frac{m'_f(z)}{(m_f(z))^2}\right).\n
\EE  
Hence, we have
\BE
	\frac{m_f'(0)}{m^2_f(0)}&=&\left(\frac{\partial q_f(s,\nu)}{\partial s}\Big|_{s=s^*_f}\right)^{-1}
	\= \ration\left(\frac{q_f(s^*_f,\nu)}{s^*_f}+s^*_f\delta\int_{\nu}^{\infty}\frac{tf(t)}{(s^*_{f}+t)^2}\dif t\right)^{-1}
		\n\\
	&=&
	\ration\left(s^*_f\delta\int_{\nu}^{\infty}\frac{tf(t)}{(s^*_{f}+t)^2}\dif t\right)^{-1}
	,	\n
\EE
where last equation is due to the fact that $q_f(s^*_f,\nu)=0$ and $s^*_f>0$. Hence, we have
\[
	\text{part 2}\ \ipt \ \frac{N}{c_N} \cdot \ration\frac{\int_{\eta_1}^{\nu}tf(t)\dif t}{s^*_f\int_{\nu}^{\infty}\frac{tf(t)}{(s^*_{f}+t)^2}\dif t},
\]
and
\[
	\text{part 3}\ \ipt \  \ration\left(s^*_f\delta\int_{\nu}^{\infty}\frac{tf(t)}{(s^*_{f}+t)^2}\dif t\right)^{-1}.
\]
This completes the proof of (ii) of the theorem.

\subsection{Proof of Part (iii)}
Suppose equation $h_f(\nu)=0$ has a solution on $(\nu_b,\infty)\bigcap \operatorname{supp}(f)$. Then by comparing the two formula in \eqref{eq:generalriskp=N} and \eqref{eq:generalminp<n1}, we just need to show $s^*_f=\frac{1}{m^*_f}<\nu^*$. Then, from \eqref{eq:generalrelation1} and \eqref{eq:generalrelation_eq2}, we have 
\BE
	\forall\ s_0\in (0,\infty),\ \  \text{if} \ \ q_f(s_0)>0, \quad \text{then} \ \ s^*_f<s_0. \n
\EE
Hence, it is sufficient to show that $q_f(\nu^*)>h_f(\nu^*)=0$. Note that $\forall \nu\geq \eta_1$
\BE
	h_f(\nu)-q_f(\nu)
	&=&
	\nu\delta\int_{\eta_1}^{\infty} \frac{tf(t)}{\nu+t}\dif t-\nu\delta\int_{\nu}^{\infty}f(t)\dif t-\delta\int_{\eta_1}^{\nu}tf(t)\dif t
		\n\\
	&=&\delta\nu \left(\int_{\nu}^{\infty} \frac{tf(t)}{\nu+t}\dif t-\int_{\nu}^{\infty}f(t)\dif t\right)+\delta\left(\int_{\eta_1}^{\nu} \frac{\nu}{\nu+t}tf(t)\dif t-\int_{\eta_1}^{\nu}tf(t)\dif t\right)
	\ < \
	0.	\n
\EE 
Then since $\nu^*>\nu_b>\eta_1$, we have $q_f(\nu^*)>h_f(\nu^*)=0$.

If equation $h_f(\nu)=0$ does not have a solution on $(\nu_b,\infty)\bigcap \operatorname{supp}(f)$, then by comparing the two formula in \eqref{eq:generalriskp=N} and \eqref{eq:generalminp<n2}, we just need to show 
\[
	\ration s^*_f
	\=
	\frac{\ration}{m^*_f}
	\ <\ \delta\int_{\eta_1}^{\infty}tf(t)\dif t
	,
\]
which is true because, due to $q_f(s^*_f)=0$, we have
\BE
	\ration s^*_f&=&s^*_f\delta\int_{\eta_1}^{\infty} \frac{tf(t)}{s^*_f+t}\dif t \ < \ \delta\int_{\eta_1}^{\infty} tf(t) \dif t. \n	
\EE

Putting everything together completes
the proof of part (iii).

 \end{appendix}

\end{document}